\documentclass[11pt]{amsart} %loads amsmath, amsfonts, amsthm

\usepackage{amssymb,cite,hyperref}
\usepackage{youngtab}
\usepackage{ytableau}
\usepackage[margin=1in]{geometry}
\usepackage{tensor,enumitem,verbatim}
\usepackage[all,cmtip]{xy}
\usepackage{todonotes}

\numberwithin{equation}{subsection}

\newtheorem{theorem}{Theorem}[section]

\newtheorem{lemma}[theorem]{Lemma}

\theoremstyle{definition}

\newtheorem{remark}[theorem]{Remark}

\title[Two Boundary Centralizer Algebras for $\mathfrak{gl}(n|m)$]{Two Boundary Centralizer Algebras for $\mathfrak{gl}(n|m)$}

\author{Jieru Zhu}
\address{Department of Mathematics \\
		University of Oklahoma \\
		Norman, OK 73019, USA}
\email{jieruzhu@buffalo.edu}

\allowdisplaybreaks

\begin{document}

\begin{abstract}
We define an action of the degenerate two boundary braid algebra $\mathcal{G}_d$ on the $\mathbb{C}$-vector space $M\otimes N\otimes V^{\otimes d}$, where $M$ and $N$ are arbitrary modules for the general linear Lie superalgebra $\mathfrak{gl}(n|m)$, and $V$ is the natural representation. When $M$ and $N$ are parametrized by rectangular hook Young diagrams, this action factors through a quotient $\mathcal{H}^{\operatorname{ext}}_d$. The irreducible summands of $M\otimes N\otimes V^{\otimes d}$ for the centralizer of $\mathfrak{gl}(n|m)$, remain irreducible once regarded as modules for this quotient $\mathcal{H}^{\operatorname{ext}}_d$.
\end{abstract}

\maketitle

\section{Introduction}
Schur-Weyl duality states that the action of the general linear group $GL(V)$ and that of the symmetric group $\Sigma_d$ fully centralize each other on the tensor space $V^{\otimes d}$. As a consequence, the irreducible representations of $GL(V)$ and $\Sigma_d$ that occur as direct summands in $V^{\otimes d}$ are paramatrized by the same index set, the set of Young diagrams with $d$ total boxes. The same diagram gives the highest weight for an irreducible $GL(V)$-module, as well as the cycle type of a conjugacy class in $\Sigma_d$.

%This provides a key link between the representation theory of $GL_n$ and that of $\Sigma_d$. Specifically, the irreducible representations of $GL_n$ that occur as summands of $V^{\otimes d}$, or the polynomial representations, are precisely those of integral dominant highest weight. On the other hand, the irreducible representations of $\Sigma_d$, are parametrized by cycle types or partitions of the integer $d$, which can be visualized as Young diagrams. By viewing a Young diagram as a highest weight for $GL_n$, these two families of representations share the same index set, as a consequence of the double centralizer theorem. 

Various generalizations have been made over the past century. More recently, Arawaka-Suzuki \cite{AS} noticed that by changing the underlying space $V^{\otimes d}$ to $X\otimes V^{\otimes d}$, where $X$ is an object in the category $\mathcal{O}$ of $\mathfrak{sl}_n$, one achieves a similar duality. In this case, the symmetric group is replaced by the graded affine Hecke algebra $H_d$ associated with $GL_n$. In particular, there is a functor from $\mathcal{O}(\mathfrak{sl}_n)$ to the category of finite dimensional representations of $H_d$. This can be considered the one boundary case.

Inspired by this result and related literature, Daugherty \cite{D} studied the two boundary case, where the underlying space is $M\otimes N\otimes V^{\otimes d}$. The modules $M$ and $N$ are taken to be highest weight representations of the Lie algebra $\mathfrak{gl}_n$ whose corresponding Young diagrams are rectangles. Accordingly, the extended degenerate two boundary Hecke algebra $\mathcal{H}^{\operatorname{ext}}$ acts on the vector space $M\otimes N\otimes V^{\otimes d}$, which centralizes the Lie algebra action. Moreover, $\mathcal{H}^{\operatorname{ext}}$ recovers much of the representation theory of the centralizer algebra $H=\operatorname{End}_{\mathfrak{gl}_n}(M\otimes N\otimes V^{\otimes d})$, in the following sense: all irreducible $H$-submodules that occur as summands of $M\otimes N\otimes V^{\otimes d}$, remains irreducible after the action is restricted to $\mathcal{H}^{\operatorname{ext}}$. Furthermore, one can construct irreducible modules for $\mathcal{H}^{\operatorname{ext}}$ explicitly using combinatorial tools.

This work generalizes the above story to the $\mathbb{Z}_2$-graded setting. The $\mathbb{Z}_2$-graded version of the existing results are as follows: the highest weight representation theory of the Lie superalgebra $\mathfrak{gl}(n|m)$ developed by Kac \cite{Kac} states that polynomial representations are parametrized by Young diagrams that are contained in a hook shape \cite{BR}. Given a partition, one can derive the highest weight of the associated representation using an explicit combinatorial manipulation \cite{BK}, this establishes a connection between the highest weight theory and the combinatorial construction of polynomial representations. 

The tensor product of two irreducible representations can be understood via multiplication of their characters, which are given by hook Schur functions. According to Remmel \cite{Rem}, this decomposition adimts the same coeffcients as multiplication of usual Schur functions, known as the Littlewood-Richardson coefficients, which in turn captures the decomposition of tensor products of $\mathfrak{gl}_n$-modules. Therefore, all combinatorial results about Littlewood-Richardson coefficients can be used for the $\mathbb{Z}_2$-graded case, including the multiplicity free result \cite{Stem} and the details of decomposition of two rectangles \cite{Stan}.

In Section 2, we will review the highest weight theory of the Lie superalgebra $\mathfrak{gl}(n|m)$, and state the explicit method in \cite{BK} to derive the associated partition. 

In Section 3, we will review the definition of the degenerate two boundary braid algebra $\mathcal{G}_d$ used in \cite{D}, as the algebra generated by the polynomial rings $\mathbb{C}[x_1,\dots,x_d]$, $\mathbb{C}[y_1,\dots,y_d]$, $\mathbb{C}[z_0,\dots,z_d]$, the group algebra $\mathbb{C}\Sigma_d$ of the symmetric group, subject to a list of relations mentioned in \cite{D}. Let $M$ and $N$ be representations for $\mathfrak{gl}(n|m)$ in the category $\mathcal{O}$, the following is true.
\begin{theorem}
There is a well defined action 
\begin{align*}
\Phi: \mathcal{G}_d \to \operatorname{End}_{\mathfrak{gl}(n|m)}(M\otimes N \otimes V^{\otimes d})
\end{align*}
\end{theorem}

Assume further that $M=L(\alpha)$ and $N=L(\beta)$ are irreducible representations with highest weight $\alpha$, $\beta$ labeled by rectangles $(a^p)$, $(b^q)$, where $(a^p)$ is a rectangle with $p$ rows of $a$ boxes and $(b^q)$ a rectangle with $q$ rows of $b$ boxes. The extended degenerate two boundary Hecke algebra $\mathcal{H}^{\operatorname{ext}}$ is defined as the quotient of $\mathcal{G}_d$ via an explicit set of relations based on the choice of $a,b,p,q$. Using formulas for computing the eigenvalues of the action of the polynomial generators $x_1$ and $y_1$, we show the following.  
\begin{theorem}
There is a well defined action 
\begin{align*}
\Psi: \mathcal{H}^{\operatorname{ext}} \to \operatorname{End}_{\mathfrak{gl}(n|m)}(L(\alpha)\otimes L(\beta) \otimes V^{\otimes d})
\end{align*}
\end{theorem}

In Section 4, we will review the double centralizer theorem and explain how it relates to the application of combinatorial tools such as Bratteli graphs and Young tableaux. In particular, 
\begin{align*}
L(\alpha)\otimes L(\beta)\otimes V^{\otimes d}\simeq \bigoplus_{\lambda} L(\lambda)\otimes \mathcal{L}^{\lambda}
\end{align*}
as a ($\mathfrak{gl}(n|m)$, $\mathcal{H}$)-bimodule, where $\mathcal{H}=\operatorname{End}_{\mathfrak{gl}(n|m)}(L(\alpha)\otimes L(\beta) \otimes V^{\otimes d})$ is the centralizer algebra. 

Let $\lambda$ be a partition in the above decomposition, and $\mathcal{P}_0$ the set of partitions occuring in the decomposition of $L(\alpha)\otimes L(\beta)$. For $\mu\in\mathcal{P}_0$, a semistandard tableau of shape $\lambda/\mu$ is a filling of the boxes in $\lambda$ but not in $\mu$, with integers $1$, $2$, $3$, \dots, so that numbers increase from left to right and from top to bottom.
\begin{theorem}
$\mathcal{L}^{\lambda}$ admits a basis $\{v_T\}$, where the indices $T$ are parametrized by all semistandard tableaux of shapes $\lambda/\mu$, for all $\mu\in \mathcal{P}_0$ that is contained in $\lambda$. Moreover,
\begin{align*}
z_i.v_T&=c(i)v_T \hspace{.5 in} 1\leq i \leq d
\end{align*} 
where $c(i)$ is the content of the box $i$.

Let $\mathfrak{B}$ be the set of boxes in the rows $p+1$ and below in $\mu$, 
\begin{align*}
z_0. v_T=
(qab+\displaystyle\sum_{b\in\mathfrak{B}}(2c(b)-(a-p+b-q)))v_T.
\end{align*}
\end{theorem}

These eigenvalues of the polynomial generators determine the action of $\mathcal{H}^{\operatorname{ext}}$, similar to the result in the $\mathfrak{gl}_n$ setting \cite{D}. Based on the information, we can construct representations $\mathcal{L}_{\lambda}$ of $\mathcal{H}^{\operatorname{ext}}$ that occur as the restrictions of $\mathcal{L}^{\lambda}$ from the centralizer algebra $\mathcal{H}$ to $\mathcal{H}^{\operatorname{ext}}$. They are indeed the same as those in \cite{D}, hence are irredicuble after restriction. In this sense, $\mathcal{H}^{\operatorname{ext}}$ also yields a large subalgebra of the centralizer $\operatorname{End}_{\mathfrak{gl}(n|m)}(L(\alpha)\otimes L(\beta)\otimes V^{\otimes d})$.

\section{The Lie Superalgebra}
\subsection{$\mathfrak{gl}(n|m)$ and a Casimir element}
Fix $m,n\in \mathbb{Z}_{\geq 0}$, the Lie superalgebra $\mathfrak{g}=\mathfrak{gl}(n|m)$ is the $\mathbb{C}$-vector space with a $\mathbb{Z}_2$-grading, whose elements consist of $(n+m)\times (n+m)$ matrices. Here, $\mathfrak{g}=\mathfrak{g}_{\overline{0}}\oplus \mathfrak{g}_{\overline{1}}$, where
\begin{align*}
\mathfrak{g}_{\overline{0}}=\left\{\begin{bmatrix}
A &0 \\ 0 & D 
\end{bmatrix}| A\in \operatorname{Mat}_{n,n}(\mathbb{C}), D\in \operatorname{Mat}_{m,m}(\mathbb{C})\right\},\\
\mathfrak{g}_{\overline{1}}=\left\{\begin{bmatrix}
0 &B \\ C & 0 
\end{bmatrix}| B\in \operatorname{Mat}_{n,m}(\mathbb{C}), C\in \operatorname{Mat}_{m,n}(\mathbb{C})\right\}.
\end{align*}
We denote by $\overline{x}$ the parity of $x$, which is $\overline{0}$ if $x\in \mathfrak{g}_{\overline{0}}$ and $\overline{1}$ if $x\in \mathfrak{g}_{\overline{1}}$. Let $I=\{1,2,\dots,m+n\}$ and $\overline{\cdot}: I\to \mathbb{Z}_2$ be the map such that $\overline{i}=\overline{0}$ if $1\leq i\leq n$, $\overline{i}=\overline{1}$ if $n+1\leq i\leq m+n$.

Define the Lie super bracket on $\mathfrak{g}$ as 
\begin{align*}
[x,y]=xy-(-1)^{\overline{x}\cdot \overline{y}} yx,
\end{align*}
for homogeneous elements $x,y\in \mathfrak{g}$. 

Let the Cartan subalgebra of $\mathfrak{g}$ be

\begin{align*}
\mathfrak{h}=\left\{\begin{bmatrix}
A &0 \\ 0 & D 
\end{bmatrix} \in \mathfrak{g}_{\overline{0}} | A,D: \text{diagonal matrices} \right\}
\end{align*}
and $d_i$ be the elementary diagonal matrix with a single $1$ in the $i$-th diagonal entry and $0$'s elsewhere. Let $\epsilon_i: \mathfrak{h}\to \mathbb{C}$ be the dual of $d_i$. $\mathfrak{g}$ has a root space decomposition 
\begin{align*}
\mathfrak{g}=\mathfrak{h}\oplus \displaystyle\bigoplus _{\alpha\in \Phi}\mathfrak{g}_{\alpha}, \hspace{.2 in}\mathfrak{g}_{\alpha}=\{x\in \mathfrak{g}| [h,x]=\alpha(h)x, \forall h\in \mathfrak{h}\}
\end{align*}
where $\Phi=\{\epsilon_i-\epsilon_j|1\leq i,j\leq n+m, i\neq j\}$ is the set of roots , and $\Phi^+=\{\epsilon_i-\epsilon_j|1\leq i<j \leq n+m\}$ is the set of positive roots. A root $\epsilon_i-\epsilon_j$ is declaired to have parity $\overline{i}+\overline{j}$. Note this parity agrees with the parity of $\mathfrak{g}$: $\mathfrak{g}_{\alpha}\subset \mathfrak{g}_{\overline{\alpha}}$. Denote by $\Phi^+_{\overline{0}}$ the set of even positive roots and $\Phi^+_{\overline{1}}$ the set of odd positive roots. Let 
\begin{align*}
2\rho = \displaystyle\sum_{\alpha\in \Phi^+_{\overline{0}}}\alpha- \displaystyle\sum_{\alpha\in \Phi^+_{\overline{1}}}\alpha 
\end{align*}

Define a bilinear form $\langle,\rangle$ on the dual of the Cartan $\mathfrak{h}^*$:
\begin{align*}
\langle,\rangle : \mathfrak{h}^* \times \mathfrak{h}^* &\to \mathfrak{h}^{*} \\
(\epsilon_i,\epsilon_j) &\mapsto (-1)^{\overline{i}}\delta_{ij}
\end{align*}

Let $V$ be the $\mathbb{C}$-vector space consisting of column vectors of height $m+n$ that admits the following $\mathbb{Z}_2$ grading: let $e_1,\dots,e_{n+m}$ be the standard basis of $V$, declare $e_1,\dots,e_n$ to be even and $e_{n+1},\dots,e_{n+m}$ to be odd, and denote by $\overline{v}$ the parity of a homogeneous vector $v\in V$. We can regard $V$ as a $\mathfrak{gl}(n|m)$-super module by matrix multiplication under this ordered basis. For the rest of this article we will refer to supermodules as modules.

Let $U(\mathfrak{g})$ be the universal enveloping algebra associated to $\mathfrak{g}$. The coalgebra structure on $U(\mathfrak{g})$ is given by the following comultiplication map 
\begin{align*}
\Delta: U(\mathfrak{g}) &\to U(\mathfrak{g}) \otimes U(\mathfrak{g}) \\
x &\mapsto  x\otimes 1 +1\otimes x
\end{align*}
for $x\in \mathfrak{g}$.
In particular, this gives a module structure on the tensor product $V^{\otimes d}$, therefore $\mathfrak{gl}(n|m)$ acts on $V^{\otimes d}$ via
\begin{align*}
&x. (e_{i_1}\otimes \cdots \otimes e_{i_d})\\
=&(x.e_{i_1})\otimes \cdots \otimes e_{i_d}+
(-1)^{\overline{x}\cdot\overline{e_{i_1}}}e_{i_1} \otimes (x.e_{i_2})\otimes \cdots \otimes e_{i_d} +\\
&\cdots+(-1)^{\overline{x}\cdot(\overline{e_{i_1}}+\overline{e_{i_2}}+\cdots+\overline{e_{i_{d-1}}})}e_{i_1}\otimes \cdots \otimes (x.e_{i_d})
\end{align*} 
for a homogeneous element $x\in \mathfrak{gl}(n|m)$, and $1\leq i_1,\dots,i_d \leq n+m$.

Let $E_{ij}$ be the matrix unit in $\mathfrak{g}$ with a single $1$ in the $(i,j)$-entry and $0$'s everywhere else ($1\leq i,j \leq n+m$). Let 
\begin{align*}
\kappa=\displaystyle\sum_{i,j=1}^{n+m}(-1)^{\overline{j}}E_{ij}E_{ji} \in U(\mathfrak{gl}(n|m))
\end{align*}
be the Casimir element of degree 2, it is known to be central (e.g. see \cite{Kuj}.)

\subsection{Polynomial representations and Young diagrams}\label{young}

Let $X=\displaystyle\bigoplus_{i=1}^{n+m}\mathbb{Z}\epsilon_i$ be the weight lattice in $\mathfrak{g}$, $X_{\geq 0}=\displaystyle\bigoplus_{i=1}^{n+m}\mathbb{Z}_{\geq 0}\epsilon_i$ the positive weight lattice. A weight $\lambda=(\lambda_1,\dots,\lambda_{n+m}) \in X_{\geq 0}$ is dominant if 

\begin{align}
&\lambda_1\geq \lambda_2\geq \cdots \geq \lambda_n  \notag \\
&\lambda_{n+1} \geq \lambda_{n+2} \geq \cdots \geq \lambda_{n+m} \label{dominant}
\end{align}

By Sergeev\cite{Ser3}, Berele-Regev\cite{BR}, for each $\lambda\in X_{\geq 0}$ that is dominant, there is a unique irreducible highest weight representation $L(\lambda)$ with highest weight $\lambda$. We call $L(\lambda)$ a polynomial representation of $\mathfrak{gl}(n|m)$ if it is an irreducible summand of $V^{\otimes d}$ for some $d$. In particular, if $l$ is the number of nonzero entires in $\lambda_{n+1}$,\dots, $\lambda_{n+m}$, $L(\lambda)$ is polynomial if and only if $\lambda_{n}\geq l$. We will refer to this condition, combined with the dominant condition in \ref{dominant} as condition (H1).

By a partition $\lambda$ with at most $m+n$ parts we mean a decreasing sequence of nonnegative integers $\lambda_1 \geq \lambda_2 \geq \cdots \geq \lambda_{n+m}$. Given such a partition $\lambda$, we call it a hook partition if it satisfies the condition 
\begin{align*}
\lambda_{n+1} &\leq m \tag{H2}
\end{align*}

It is convenient to draw partitions as Young diagrams.  For a partition $\lambda=(\lambda_1,\lambda_2,\dots)$ associate a Young diagram consists of $\lambda_1$ boxes in the first row, $\lambda_2$ boxes in the second row, \dots, etc. We will use the terms partitions and Young diagrams interchangeably.

Define a map $\lambda\mapsto \overline{\lambda}$ from the set of hook Young diagrams to the set of integral dominant weights satisfying (H1) as follows: divide the Young diagram for $\lambda$ into two subdiagrams: the even part $\lambda_{\operatorname{even}}$ that is the Young diagram consisting of the rows $1$ through $n$, and the odd part $\lambda_{\operatorname{odd}}$ that is the Young diagram consisting of the rows $n+1$ through $n+m$. We obtain the diagram $\overline{\lambda}$ by pasting the transpose $\lambda_{\operatorname{odd}}^T$ of $\lambda_{\operatorname{odd}}$ right underneath $\lambda_{\operatorname{even}}$. 

For example, In the case $n=2, m=3$, let $\lambda$ be the following diagram
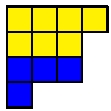
\begin{align*}
\ytableausetup{smalltableaux}
\begin{ytableau}
*(yellow)&  *(yellow)& *(yellow)&*(yellow)\\
*(yellow)&*(yellow)&*(yellow)\\
*(blue)&*(blue)&*(blue)\\
*(blue)
\end{ytableau}
\end{align*}
Then
\begin{align*}
\lambda_{\text{even}}=\begin{ytableau}
*(yellow)&  *(yellow)& *(yellow)&*(yellow)\\
*(yellow)&*(yellow)&*(yellow)
\end{ytableau} \hspace{.2 in}
\lambda_{\text{odd}}&=\begin{ytableau}
*(blue)&*(blue)&*(blue)\\
*(blue)
\end{ytableau}\hspace{.2 in}
\lambda^T_{\text{odd}}=\begin{ytableau}
*(blue)&*(blue)\\
*(blue)\\
*(blue)
\end{ytableau}
\end{align*}
Therefore $\overline{\lambda}=(4,3,2,1,1)$ and can again be represented by the following Young diagram:
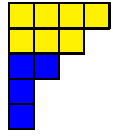
\begin{align*}
\begin{ytableau}
*(yellow)&  *(yellow)& *(yellow)&*(yellow)\\
*(yellow)&*(yellow)&*(yellow)\\
*(blue)&*(blue)\\
*(blue)\\
*(blue)
\end{ytableau}
\end{align*}

The $\overline{\cdot}$ has image in the target set since if $\lambda$ satisfies Condition (H2), then $\overline{\lambda}$ must have at most $n+m$ parts, and both the even and odd parts are weakly decreasing nonnegative integers. On the other hand, it has an obvious inverse $\overline{\cdot}:\mu\to \overline{\mu}$, sending a weight $\mu$ satisfying condition (H1) to a hook Young diagram, using the same procedure. The resulting shape is still a Young diagram, as guaranteed by condition (H1). The fact that it is contained in the $(n,m)$-hook comes from the fact that $\mu$ has at least $n+m$ nonzero entries. Therefore, $\overline{\cdot}$ establishes a bijection between the two sets.
%The resulting weight $\overline{\lambda}$ is indeed in the target set, since if $\lambda$ satisfies Condition (H2), then $\overline{\lambda}$ must have at most $n+m$ parts. It is surjective since if $\mu\in X^+$  satisfies Condition (H1), then by applying the same operations as those defined in the map $\overline{\cdot}$, one can obtain its preimage $\gamma$ with $\overline{\gamma}=\mu$. It is trivial to check the map is injective, therefore $\lambda \mapsto \overline{\lambda}$ gives a bijection between the two sets. Since the inverse map is given by performing the same operations on a weight, we sometimes write $\overline{\cdot}$ for the inverse map if the context is clear.

Let $b$ be a box in the Young diagram $\lambda$ that sits in the $l$-th column and $r$-th row. The content $c(b)$ of the box $b$ is defined as $c(b)=l - r$.

\subsection{The Littlewood-Richardson coefficients}\label{lr}
Let $\lambda,\mu\in X_{\geq 0}$ be dominant weights satisfying Condition (H1). Since the modules $L(\lambda)$, $L(\mu)$ are polynomial, $L(\lambda)$ occurs as a direct summand of $V^{\otimes s}$ for some $s\in\mathbb{Z}$, and so is $L(\mu)$ for some $V^{\otimes t}$. The module $L(\lambda)\otimes L(\mu)$ is therefore a direct summand of $V^{\otimes (s+t)}$, and hence semisimple. It decomposes as a $\mathfrak{g}$-module:

\begin{align*}
L(\lambda)\otimes L(\mu) \simeq \displaystyle\bigoplus_{\gamma} L(\gamma) ^{\oplus \overline{c}_{\lambda,\mu}^{\gamma}}
\end{align*}

Where the direct sum is over all dominant weights in $X^+$ satisfying Condition (H1), and $\overline{c}_{\lambda,\mu}^{\gamma}$ is the multiplicity of $L(\gamma)$ in $L(\lambda)\otimes L(\mu) $.

The usual Littlewood-Richardson coefficients are defined as follows: Let $\lambda$, $\mu$ be usual partitions. Let $s_{\lambda}$, $s_{\mu}$ be the usual Schur functions associated to $\lambda$ and $\mu$, as elements in the polynomial ring of infinitely many variables. It is well-known that 
\begin{align*}
s_{\lambda}s_{\mu}=\displaystyle\sum_{\gamma:|\gamma|=|\mu|+|\lambda|}c_{\lambda,\mu}^{\gamma}s_{\gamma} 
\end{align*}
where $|\lambda|$ and $|\mu|$ denote the sum of the entries in $\lambda$ and $\mu$. The coefficients $c_{\lambda,\mu}^{\gamma}$ are called the Littlewood Richardson coefficients and have combinatorial interpretations.

By \cite{Rem}, for hook partitions $\lambda$, $\mu$, one can associate hook Schur functions $s'_{\lambda}$ and $s'_{\mu}$ in terms of combinatorial algorithms. Similar to the previous case,
\begin{align*}
s'_{\lambda}s'_{\mu}=\displaystyle\sum_{\gamma:|\gamma|=|\mu|+|\lambda|}{c}_{\lambda,\mu}^{\gamma}s'_{\gamma} 
\end{align*}
where the coefficients are the same as the Littlewood-Richardson coefficients mentioned above. 

Since hook Schur functions correspond to the characters of representations of Lie superalgebras, and multiplication of characters correspond to the tensor products of representations, it follows that 
\begin{align*}
\overline{c}^{\gamma}_{\lambda,\mu}=c^{\gamma}_{\lambda,\mu}
\end{align*}

\section{The Extended Two Boundary Hecke Algebra}

\subsection{The Extended Two Boundary Braid Algebra}

As mentioned in \cite[Section~2]{D}, the degenerate two boundary braid algebra $\mathcal{G}_d$ is a quotient of the $\mathbb{C}$-algebra 

\begin{align*}
\mathbb{C}[x_1,\dots,x_d] \otimes  \mathbb{C}[y_1,\dots,y_d]\otimes  \mathbb{C}[z_0,\dots,z_d] \otimes \mathbb{C}\Sigma_d
\end{align*}
subject to further relations below. Here $\Sigma_d$ is the symmetric group on $d$ letters and denote by $t_i(1\leq i\leq d-1)$ the simple transpositions in $\Sigma_d$. 

Define
\begin{align*}
m_j&=\sum_{i=1}^{j-1}m_{i,j}\\
m_{i,j}&=\begin{cases}
x_{i+1}-t_ix_it_i \hspace{.2 in } j=i+1\\
t_{(i j-1)}m_{j-1,j}t_{(i j-1)} \hspace{.2 in} j\neq i+1
\end{cases}
\end{align*}
and $t_{(i j)}=t_it_{i+1}\cdots t_{j-2} t_{j-1} t_{j-2}\cdots t_1$ corresponds to the swap in $\Sigma_d$ that interchanges $i$ and $j$. 

The relations are as follows:
\begin{align*}
x_it_j=t_jx_i, \hspace{.2 in} y_it_j=t_jy_i, \hspace{.2 in } z_it_j=t_jz_i, \hspace{.2 in} i\neq j,j+1  \\
\tag{R1}\\
(z_0+z_1+\cdots+z_i)x_j = x_j(z_0+z_1+\cdots+z_i)  \hspace{.2 in} i\geq j \notag \\
(z_0+z_1+\cdots+z_i)y_j = y_j(z_0+z_1+\cdots+z_i)  \hspace{.2 in} i\geq j\\
 \tag{R2}\\
t_i(x_i+x_{i+1})=(x_i+x_{i+1})t_i,\hspace{.2 in}t_i(y_i+y_{i+1})=(y_i+y_{i+1})t_i \hspace{.2 in} 1\leq i \leq d-1 \\
\tag{R3}\\
t_it_{i+1}(x_{i+1}-t_ix_{i}t_i)t_{i+1}t_i=x_{i+2}-t_{i+1}x_{i+1}t_{i+1} \hspace{.2 in} \\
t_it_{i+1}(y_{i+1}-t_iy_{i}t_i)t_{i+1}t_i=y_{i+2}-t_{i+1}y_{i+1}t_{i+1} \hspace{.2 in}
 1\leq i\leq d-2 \notag \\
 \tag{R4}\\
 x_{i+1}-t_ix_it_i=y_{i+1}-t_iy_it_i \hspace{.2 in} 1\leq i\leq d-1 \\
 \tag{R5}\\
z_i=x_i+y_i-m_i \hspace{.2 in} 1\leq i\leq d \tag{R6}
\end{align*}

 Recall the comultiplication map $\Delta$ and Casimir element $\kappa$ defined in Section 1.
\begin{align*}
\Delta(\kappa)=1\otimes \kappa + \kappa \otimes 1 + 2\gamma
\end{align*}
where 
\begin{align*}
\gamma = \displaystyle\sum_{p,q=1}^{n+m}(-1)^{\overline{q}}E_{pq}\otimes E_{qp} \in \mathfrak{gl}(n|m) \otimes \mathfrak{gl}(n|m)
\end{align*}

Let $M$, $N$ be $\mathfrak{gl}(n|m)$-modules in the category $\mathcal{O}$, the goal of this section is to define an action of $\mathcal{G}_d$ on $M\otimes N\otimes V^{\otimes d}$ that commutes with the action of $\mathfrak{g}$ using the above elements.

Define the following elements in $U(\mathfrak{gl}(n|m))^{\otimes d+2}$ as follows:

\begin{tabular}{c c}
$\kappa_{X,i}$ & $\kappa$ acting on $X$ and the first $i$ copies of $V$, $0\leq i\leq d$\\
$\kappa_i$ & $\kappa$ acting on the $i$-th copy of $V$\\
 $\gamma_{X,i}$ & $\gamma$ acting on $X$ and the $i$-th copy of $V$, $1\leq i\leq d$\\
 $\gamma_{i,j}$ & $\gamma$ acting on the $i$-th and $j$-th copy of $V$,  $1\leq i<j\leq d$
\end{tabular}\\

Specifically, viewing $M\otimes N\otimes V^{\otimes d}$ as a module for $U(\mathfrak{gl}(n|m))^{\otimes d+2}$, these are elements in $U(\mathfrak{gl}(n|m))^{\otimes d+2}$ given as follows:

\begin{align*}
\kappa_i & = 1_M\otimes 1_N \otimes 1^{\otimes i-1}\otimes \kappa \otimes 1^{\otimes d-i}\\
\kappa_{M,i}&=\kappa \otimes 1_N \otimes 1_V^{\otimes d} +  \kappa_1+\cdots+\kappa_i \\
\kappa_{N,i}&=1_M \otimes \kappa \otimes 1_V^{\otimes d} +  \kappa_1+\cdots+\kappa_i \\
\kappa_{M\otimes N,i}&=\kappa \otimes 1_N \otimes 1_V^{\otimes d} +1_M \otimes \kappa \otimes 1_V^{\otimes d} +  \kappa_1+\cdots+\kappa_i \\
\gamma_{M,i}&= \sum_{p,q}(-1)^{\overline{q}}E_{pq}\otimes 1_N\otimes 1_V^{i-1} \otimes E_{qp} \otimes 1_V^{\otimes d-i}\\
\gamma_{N,i}&=\sum_{p,q}(-1)^{\overline{q}} 1_M\otimes E_{pq}\otimes 1_V^{i-1} \otimes E_{qp} \otimes 1_V^{\otimes d-i}\\
\gamma_{M\otimes N, i} &= \gamma_{M,i}+\gamma_{N,i}\\
\gamma_{i,j} & = \sum_{p,q}(-1)^{\overline{q}} 1_M\otimes 1_N \otimes 1_V^{\otimes i-1}\otimes E_{pq}\otimes 1_V^{\otimes j-i-1} E_{qp} \otimes 1_V^{\otimes d-j}
\end{align*}

%Here $X$ stands for any of the symbols $M$, $N$ or $M\otimes N$. In particular, Let $\Delta^j: U(\mathfrak{g})\to U(\mathfrak{g})^{\otimes j}$ be the map $\Delta^j= (1^{\otimes (j-1)}\otimes \Delta)\cdots (1\otimes \Delta)\circ \Delta$, where $1$ stands for the identity map on $U(\mathfrak{g})$. Then $\Delta^j$ defines an action of $U(\mathfrak{g})$ on any $\mathfrak{g}$-module $W_1\otimes \cdots\otimes W_j$, where each $W_i$ is a $\mathfrak{g}$-module. Furthermore, let $J=\{M,N,1,2,\dots,d\}$ be the index set, for any $P=\{P_1,\dots,P_s\}\subset J$,  define $\iota_P: U(\mathfrak{g})^{\otimes |P|} \to U(\mathfrak{g})^{\otimes (d+2)}$ to be the injective map 

%\begin{align*}
%f_{1}\otimes \cdots \otimes f_{s} \mapsto g_M \otimes g_N \otimes \cdots \otimes g_{d+2}
%\end{align*}
%where $g_i=f_{i}$ if $i\in P$, and $g_i=1$ otherwise.

%Then the elements mentioned earlier are explicit elements in $U(g)^{\otimes (d+2)}$. For example,
%\begin{align*}
%\kappa_{M,i}&= \iota_{M,1,2,\dots,i} \circ \Delta^{i+1}(\kappa)\\
%\kappa_{N,i}&= \iota_{N,1,2,\dots,i} \circ \Delta^{i+1}(\kappa)\\
%\kappa_{M\otimes N,i}&= \iota_{M,N,1,2,\dots,i} \circ \Delta^{i+2}(\kappa)\\
%\gamma_{M,i}&= \iota_{M,i}(\gamma)\\
%\gamma_{N,i}&= \iota_{N,i}(\gamma) \\
%\gamma_{i,j}&= \iota_{i,j}(\gamma) \\
%\gamma_{M\otimes N,i}&= \iota_{M,N,i}(\Delta \otimes 1)(\gamma)\\
%\end{align*}

Here, the summation is over all pairs of integers $(p,q)$ with $1\leq p,q\leq n$. Notice $\kappa$ acts on $V$ by a scalar since $\kappa$ is central and $V$ is irreducible, and denote this scalar by $\kappa_V$. We claim that there is an action of $\mathcal{G}_k$ on $M\otimes N\otimes V^{\otimes d}$.

\begin{theorem}\label{action}
There is a well-defined algebra homomorphism 
\begin{align*}
\rho: \mathcal{G}_k \to \operatorname{End}_{\mathfrak{gl}(n|m)}(M\otimes N\otimes V^{\otimes d})
\end{align*}
Given by 
\begin{align*}
t_i & \mapsto 1_M\otimes 1_N \otimes 1_V^{\otimes i-1} \otimes \sigma  \otimes 1_V^{\otimes d-i-1}\\
x_i &\mapsto \frac{1}{2}(\kappa_{M,i} -\kappa_{M,i-1}) \\
y_i &\mapsto \frac{1}{2}(\kappa_{N,i}- \kappa_{N,i-1})\\
z_i&\mapsto \frac{1}{2}(\kappa_{M\otimes N,i}-\kappa_{M\otimes N,i-1}+\kappa_V)\\
z_0 &\mapsto \frac{1}{2}(\kappa_{M\otimes N}-\kappa_M-\kappa_N)
\end{align*}
where $\sigma$ is the signed swap on $V\otimes V$ with
\begin{align*}
\sigma.( v\otimes w)=(-1)^{\overline{v}\cdot \overline{w}}w \otimes v
\end{align*}
for homogeneous $v,w\in V$.
\end{theorem}

\begin{proof}
First let us verify the relation (R1).

To verify $t_jx_i=x_it_j$ for $i\neq j,j+1$, notice the action of $x_i$ is given by the element $U(\mathfrak{gl}(n|m))^{\otimes d+2}$ equal to the following:
\begin{align*}
\frac{1}{2}(\kappa_{M,i}-\kappa_{M,i-1})= \frac{1}{2}\kappa_V +  \gamma_{M,i} +\gamma_{1,i} +\cdots +\gamma_{i-1,i}.
\end{align*}

When $k<i<j$, it is easy to see that $\gamma_{k,i}t_j=t_j\gamma_{k,i}$, $\forall k>i$ or when $k$ is the symbol $M$. 

When $j \leq i-2$, we claim that each of $\kappa_V$, $\gamma_{M,i}$, $\gamma_{k,i}$ commutes with $t_j$ for $k\neq j,j+1$, and that 
\begin{align}
(\gamma_{j,i}+\gamma_{j+1,i})t_j=t_j(\gamma_{j,i}+\gamma_{j+1,i}). \label{action1}
\end{align}

In particular, since $\kappa$ is even, the action of $\kappa_i$ on any given tensor introduces no signs, and it commutes with the signed swap map. When $k\neq j,j+1$, $t_j$ commutes with $\gamma_{k,i}$. We discuss by cases: when $k<j$,

\begin{align*}
&t_j \gamma_{k,i} (v_M\otimes v_N\otimes v_1\otimes \cdots \otimes v_d)  \\
=& t_j (\displaystyle\sum_{p,q}(-1)^{\overline{q}}(-1)^{(\overline{v_k}+\cdots+\overline{v_{i-1}})(\overline{p}+\overline{q})}\\
& \cdot v_M\otimes \cdots \otimes  E_{pq}v_k\otimes\cdots \otimes v_j\otimes v_{j+1}\otimes  \cdots\otimes E_{qp}v_i\otimes \cdots\otimes v_d)\\
=& \displaystyle\sum_{p,q}(-1)^{\overline{q}}(-1)^{(\overline{v_k}+\cdots+\overline{v_{i-1}})(\overline{p}+\overline{q})+\overline{v_j}\cdot\overline{v_{j+1}}}\\
& \cdot v_M\otimes \cdots \otimes  E_{pq}v_k\otimes \cdots\otimes v_{j+1}\otimes v_j \otimes   \cdots\otimes E_{qp}v_i\otimes \cdots\otimes v_d.
\end{align*}
On the other hand,
\begin{align*}
&\gamma_{k,j}t_j(v_M\otimes v_N\otimes v_1\otimes \cdots\otimes v_d)\\
=&\gamma_{k,j}\sum_{p,q}(-1)^{\overline{q}}(-1)^{\overline{v_j}\cdot \overline{v_{j+1}}}\\
&\cdot v_M\otimes \cdots \otimes v_k \otimes \cdots \otimes v_{j+1}\otimes v_j\otimes \cdots \otimes v_i \otimes \cdots\otimes v_d\\
=& \displaystyle\sum_{p,q}(-1)^{\overline{q}}(-1)^{(\overline{v_k}+\cdots+\overline{v_{i-1}})(\overline{p}+\overline{q})+\overline{v_j}\cdot\overline{v_{j+1}}}\\
& \cdot v_M\otimes \cdots \otimes  E_{pq}v_k\otimes \cdots\otimes v_{j+1}\otimes v_j \otimes  \cdots\otimes E_{qp}v_i\otimes \cdots\otimes v_d.
\end{align*}
The other case when $j+1<k$ is similar.

To see (\ref{action1}) is true, 
\begin{align*}
&t_j(\gamma_{j,i}+\gamma_{j+1,i})(v_M\otimes v_N\otimes v_1\otimes \cdots \otimes v_d)\\
=&t_j(\displaystyle\sum_{p,q}(-1)^{\overline{q}}(-1)^{(\overline{v_j}+\cdots+\overline{v_{i-1}})(\overline{p}+\overline{q})}v_M\otimes \cdots \otimes E_{pq}v_j\otimes \cdots\otimes E_{qp}v_i\otimes \cdots\otimes v_d)\\
&+t_j(\displaystyle\sum_{p,q}(-1)^{\overline{q}}(-1)^{(\overline{v_{j+1}}+\cdots+\overline{v_{i-1}})(\overline{p}+\overline{q})}v_M\otimes \cdots \otimes E_{pq}v_{j+1}\otimes \cdots\otimes E_{qp}v_i\otimes \cdots\otimes v_d)\\
=&\displaystyle\sum_{p,q}(-1)^{\overline{q}}(-1)^{(\overline{v_j}+\cdots+\overline{v_{i-1}})(\overline{p}+\overline{q})}(-1)^{(\overline{p}+\overline{q}+\overline{v_j})(\overline{v_{j+1}})}\\
&\cdot v_M\otimes \cdots \otimes v_{j+1}\otimes E_{pq}v_j\otimes \cdots\otimes E_{qp}v_i\otimes \cdots\otimes v_d\\
&+\displaystyle\sum_{p,q}(-1)^{\overline{q}}(-1)^{(\overline{v_{j+1}}+\cdots+\overline{v_{i-1}})(\overline{p}+\overline{q})}(-1)^{(\overline{v_j})(\overline{p}+\overline{q}+\overline{v_{j+1}})}\\
&\cdot v_M\otimes \cdots \otimes E_{pq}v_{j+1}\otimes v_j\otimes \cdots\otimes E_{qp}v_i\otimes \cdots\otimes v_d.
\end{align*}
On the other hand,
\begin{align*}
&(\gamma_{j,i}+\gamma_{j+1,i})t_j(v_M\otimes v_N\otimes v_1\otimes \cdots \otimes v_d)\\
=&(\gamma_{j,i}+\gamma_{j+1,i})(-1)^{\overline{v_j}\cdot \overline{v_{j+1}}}(v_M\otimes v_N\otimes v_1\otimes \cdots\otimes v_{j+1}\otimes v_j\otimes \cdots \otimes v_d)\\
=&(-1)^{\overline{v_j}\cdot \overline{v_{j+1}}}\displaystyle\sum_{p,q}(-1)^{\overline{q}}(-1)^{(\overline{v_{j+1}}+\overline{v_{j}}+\overline{v_{j+2}}+\cdots+\overline{v_{i-1}})(\overline{p}+\overline{q})}\\
&\cdot v_M\otimes \cdots \otimes E_{pq}v_{j+1}\otimes v_j\otimes \cdots\otimes E_{qp}v_i\otimes \cdots\otimes v_d\\
&+(-1)^{\overline{v_j}\cdot \overline{v_{j+1}}}\displaystyle\sum_{p,q}(-1)^{\overline{q}}(-1)^{(\overline{v_{j}}+\overline{v_{j+2}}+\cdots+\overline{v_{i-1}})(\overline{p}+\overline{q})}\\
&\cdot v_M\otimes \cdots \otimes v_{j+1}\otimes E_{pq}v_j\otimes \cdots\otimes E_{qp}v_i\otimes \cdots\otimes v_d.
\end{align*}
Hence the two actions are equal if and only if the associated signs are equal
\begin{align*}
{\overline{q}}+{(\overline{v_j}+\cdots+\overline{v_{i-1}})(\overline{p}+\overline{q})}+{(\overline{p}+\overline{q}+\overline{v_j})(\overline{v_{j+1}})}\\
-({\overline{v_j}\cdot \overline{v_{j+1}}}+{\overline{q}}+{(\overline{v_{j}}+\overline{v_{j+2}}+\cdots+\overline{v_{i-1}})(\overline{p}+\overline{q})})=0.
\end{align*}
and
\begin{align*}
{\overline{q}}+{(\overline{v_{j+1}}+\cdots+\overline{v_{i-1}})(\overline{p}+\overline{q})}+{(\overline{v_j})(\overline{p}+\overline{q}+\overline{v_{j+1}})}\\-({\overline{v_j}\cdot \overline{v_{j+1}}}+{\overline{q}}+{(\overline{v_{j+1}}+\overline{v_{j}}+\overline{v_{j+2}}+\cdots+\overline{v_{i-1}})(\overline{p}+\overline{q})})=0.
\end{align*}

Now let us check (R3) holds.

To show that relation $(x_i+x_{i+1})t_i=t_i(x_i+x_{i+1})$ holds (the calculation is similar for $(y_i+y_{i+1})t_i=t_i(y_i+y_{i+1})$,) notice 
\begin{align*}
\rho(x_i+x_{i+1})&=\frac{1}{2}(\kappa_{M,i+1}-\kappa_{M,i-1})\\
&=\kappa_V+\kappa_{i+1}+\gamma_{M,i}+\gamma_{M,i+1}+\displaystyle\sum_{k=1}^{i-1}(\gamma_{k,i}+\gamma_{k,i+1})+\gamma_{i,i+1}.
\end{align*}
We claim that each $\kappa_i+\kappa_{i+1}$, $\gamma_{M,i}+\gamma_{M,i+1}$, $\gamma_{k,i}+\gamma_{k,i+1}$ and $\gamma_{i,i+1}$ commutes with $s_i$.

The calculation for $(\kappa_i+\kappa_{i+1})t_i=t_i(\kappa_i+\kappa_{i+1})$ is straightforward, using the fact that the action of $\kappa_i$ or $\kappa_{i+1}$ introduces no signs. The calculation for $\gamma_{k,i}+\gamma_{k,i+1}$ is as follows and that for $\gamma_{M,i}+\gamma_{M,i+1}$ is very similar.
\begin{align*}
&(\gamma_{k,i}+\gamma_{k,i+1})t_i(v_M\otimes v_N\otimes v_1\otimes \cdots\otimes v_n)\\
=&(\gamma_{k,i}+\gamma_{k,i+1})(-1)^{\overline{v_i}\cdot \overline{v_{i+1}}}(v_M\otimes v_N\otimes v_1\otimes \cdots\otimes v_{i+1}\otimes v_i \otimes  \cdots\otimes v_n)\\
=&(-1)^{\overline{v_i}\cdot \overline{v_{i+1}}} \displaystyle\sum_{p,q}(-1)^{\overline{q}} \\
&((-1)^{(\overline{v_k}+\dots+\overline{v_{i-1}})(\overline{p}+\overline{q})}v_M\otimes v_N\otimes \cdots \otimes E_{pq}v_k\otimes \cdots\otimes E_{qp}v_{i+1}\otimes v_i\otimes \cdots\otimes v_d+\\
&(-1)^{(\overline{v_k}+\dots+\overline{v_{i-1}}+\overline{v_{i+1}})(\overline{p}+\overline{q})}v_M\otimes v_N\otimes \cdots \otimes E_{pq}v_k\otimes \cdots\otimes v_{i+1}\otimes E_{qp} v_i\otimes \cdots\otimes v_d).
\end{align*}
\begin{align*}
&t_i(\gamma_{k,i}+\gamma_{k,i+1})(v_M\otimes v_N\otimes v_1\otimes \cdots\otimes v_n)\\
=&t_i \displaystyle\sum_{p,q}(-1)^{\overline{q}} \\
&((-1)^{(\overline{v_k}+\dots+\overline{v_{i-1}})(\overline{p}+\overline{q})}v_M\otimes v_N\otimes \cdots \otimes E_{pq}v_k\otimes \cdots\otimes E_{qp}v_{i}\otimes v_{i+1}\otimes \cdots\otimes v_d+\\
&(-1)^{(\overline{v_k}+\dots+\overline{v_{i-1}}+\overline{v_{i}})(\overline{p}+\overline{q})}v_M\otimes v_N\otimes \cdots \otimes E_{pq}v_k\otimes \cdots\otimes v_{i}\otimes E_{qp} v_{i+1}\otimes \cdots\otimes v_d)\\
=&\displaystyle\sum_{p,q}(-1)^{\overline{q}}((-1)^{(\overline{v_k}+\dots+\overline{v_{i-1}})(\overline{p}+\overline{q})} (-1)^{(\overline{p}+\overline{q}+\overline{v_i})\overline{v_{i+1}}}\\
&v_M\otimes v_N\otimes \cdots \otimes E_{pq}v_k\otimes \cdots\otimes  v_{i+1}\otimes E_{qp}v_{i}\otimes \cdots\otimes v_d+\\
&(-1)^{(\overline{v_k}+\dots+\overline{v_{i-1}}+\overline{v_i})(\overline{p}+\overline{q})} (-1)^{(\overline{p}+\overline{q}+\overline{v_{i+1}})\overline{v_{i}}}\\
&v_M\otimes v_N\otimes \cdots \otimes E_{pq}v_k\otimes \cdots\otimes E_{qp} v_{i+1}\otimes v_{i}\otimes \cdots\otimes v_d).
\end{align*}
Compare the two pairs of coefficients:
\begin{align*}
&(\overline{v_i}\cdot \overline{v_{i+1}}+\overline{q}+(\overline{v_k}+\dots+\overline{v_{i-1}})(\overline{p}+\overline{q}))-\\
&(\overline{q}+(\overline{v_k}+\dots+\overline{v_{i-1}}+\overline{v_i})(\overline{p}+\overline{q})+(\overline{p}+\overline{q}+\overline{v_{i+1}})\overline{v_{i}})=0.
\end{align*}
and 
\begin{align*}
&(\overline{v_i}\cdot \overline{v_{i+1}}+\overline{q}+(\overline{v_k}+\dots+\overline{v_{i-1}}+\overline{v_{i+1}})(\overline{p}+\overline{q}))-\\
&(\overline{q}+(\overline{v_k}+\dots+\overline{v_{i-1}})(\overline{p}+\overline{q})+(\overline{p}+\overline{q}+\overline{v_{i}})\overline{v_{i+1}})=0.
\end{align*}

To see that $\gamma_{i,i+1}t_i=t_i\gamma_{i,i+1}$, to simplify notation, let $\underline{v}=v_M\otimes v_N\otimes v_1 \otimes \cdots\otimes v_d$, $v_{begin}=v_M\otimes v_N\otimes v_1\otimes \cdots\otimes v_{i-1}$, $v_{end}=v_{i+2}\otimes \cdots\otimes v_d$, then

\begin{align*}
&\gamma_{i,i+1}t_i \underline{v}\\
=&\gamma_{i,i+1}(-1)^{\overline{v_i}\cdot\overline{v_{i+1}}}v_{start}\otimes v_{i+1}\otimes v_i \otimes v_{end}\\
=&(-1)^{\overline{v_i}\cdot\overline{v_{i+1}}} \displaystyle\sum_{p,q}(-1)^{\overline{q}}(-1)^{(\overline{p}+\overline{q})(\overline{v_{i+1}})}
v_{start}\otimes E_{pq}v_{i+1}\otimes E_{qp}v_i \otimes v_{end}.
\end{align*}
Using the fact that all entries in $v_{start}$ have commuted past both $E_{pq}$ and $E_{qp}$ in $\gamma_{i,i+1}$.
On the other hand,
\begin{align*}
t_i\gamma_{i,i+1}\underline{v}
&=t_i \displaystyle\sum_{p,q}(-1)^{\overline{q}}(-1)^{(\overline{p}+\overline{q})\overline{v_i}}v_{start}\otimes E_{pq}v_i \otimes E_{qp}v_{i+1}\otimes v_{end}\\
&=\displaystyle\sum_{p,q}(-1)^{\overline{q}}(-1)^{(\overline{p}+\overline{q})\overline{v_i}}(-1)^{(\overline{p}+\overline{q}+\overline{v_i})(\overline{p}+\overline{q}+\overline{v_{i+1}})}v_{start}\otimes E_{qp}v_{i+1} \otimes E_{pq}v_i \otimes v_{end}\\
&=\displaystyle\sum_{p,q}(-1)^{\overline{p}}(-1)^{(\overline{p}+\overline{q})\overline{v_i}}(-1)^{(\overline{p}+\overline{q}+\overline{v_i})(\overline{p}+\overline{q}+\overline{v_{i+1}})}v_{start}\otimes E_{pq}v_{i+1} \otimes E_{qp}v_i \otimes v_{end}.
\end{align*}
where the last equality is achieved by swapping the indices $p$ and $q$. Now we compare the coefficients (or signs):
\begin{align*}
&(\overline{v_i}\cdot\overline{v_{i+1}}+\overline{q}+(\overline{p}+\overline{q})(\overline{v_{i+1}}))-(\overline{p}+(\overline{p}+\overline{q})\overline{v_i}+(\overline{p}+\overline{q}+\overline{v_i})(\overline{p}+\overline{q}+\overline{v_{i+1}}))\\
=&(\overline{v_i}\cdot\overline{v_{i+1}}+\overline{q}+(\overline{p}+\overline{q})(\overline{v_{i+1}}))-\\
&(\overline{p}+(\overline{p}+\overline{q})\overline{v_i}+(\overline{p}+\overline{q})^2+\overline{v_i}(\overline{p}+\overline{q})+\overline{v_{i+1}}(\overline{p}+\overline{q})+\overline{v_i}\cdot\overline{v_{i+1}})\\
=&\overline{q}-\overline{p}-(\overline{p}+\overline{q})^2=0
\end{align*}

Now let us verify (R2).

To see why $(z_0+z_1+\cdots+z_i)x_i=x_i(z_0+z_1+\cdots+z_i)$, notice that for an element $\kappa_{M,i}$, it is the homomorphic image of $\Delta^i(\kappa)$ under the imbedding $U(\mathfrak{gl}(n|m))^{\otimes (i+1)}\hookrightarrow U(\mathfrak{gl}(n|m))^{\otimes (d+2)}$, and is central in the image. Hence $\kappa_{M,i}$ commutes with $\kappa_{M,j}$, $\kappa_{N,j}$, $\kappa_{M,N,j}$, or the action of $\mathfrak{gl}(n|m)$. Since
\begin{align*}
\rho(z_0+z_1+\cdots+z_i)=\frac{1}{2}(\kappa_{M\otimes N,i}-\kappa_M-\kappa_N+i\kappa_V)
\end{align*}
it follows immediately that it commutes with $\rho(x_i)$.

To show the remaining of the relations, claim that 
\begin{align*}
x_{i+1}-t_ix_it_i=y_{i+1}-t_iy_it_i=\gamma_{i,i+1}
\end{align*}
To see why this is true, first claim that $\gamma_{k,i+1}=t_i\gamma_{k,i}t_i$. This equality extends to the case when $k$ is the symbol $M$, yielding $\gamma_{M,i+1}=t_i\gamma_{M,i}t_i$, and the calculation is similar to the following calculation for an integer $k$.
\begin{align*}
\gamma_{k,i+1}\underline{v}=&\displaystyle\sum_{p,q}(-1)^{\overline{q}}(-1)^{(\overline{v_{k}}+\cdots+\overline{v_{i}})(\overline{p}+\overline{q})}\\
&v_M\otimes \cdots \otimes E_{pq}v_k\otimes\cdots\otimes v_{i}\otimes E_{qp}v_{i+1}\otimes \cdots\otimes v_d.\\
t_i\gamma_{k,i}t_i\underline{v}
=&t_i\gamma_{k,i} (-1)^{\overline{v_i}\cdot\overline{v_{i+1}}}v_M\otimes \cdots \otimes v_k\otimes\cdots\otimes v_{i+1}\otimes v_{i}\otimes \cdots\otimes v_d\\
=&(-1)^{\overline{v_i}\cdot \overline{v_{i+1}}}t_i\displaystyle\sum_{p,q}(-1)^{\overline{q}}(-1)^{(\overline{v_{k}}+\cdots+\overline{v_{i-1}})(\overline{p}+\overline{q})}\\
&\cdot v_M\otimes \cdots \otimes E_{pq}v_k\otimes\cdots\otimes E_{qp} v_{i+1}\otimes v_{i}\otimes \cdots\otimes v_d\\
=&(-1)^{\overline{v_i}\cdot\overline{v_{i+1}}}\displaystyle\sum_{p,q}(-1)^{\overline{q}}(-1)^{(\overline{v_{k}}+\cdots+\overline{v_{i-1}})(\overline{p}+\overline{q})}(-1)^{(\overline{p}+\overline{q}+\overline{v_{i+1}})\overline{v_i}}\\
&\cdot v_M\otimes \cdots \otimes E_{pq}v_k\otimes\cdots\otimes  v_{i}\otimes E_{qp} v_{i+1}\otimes\cdots\otimes v_d.
\end{align*}
and the signs match:
\begin{align*}
&{\overline{q}}+{(\overline{v_{k}}+\cdots+\overline{v_{i}})(\overline{p}+\overline{q})}\\
&-({\overline{v_i}\cdot\overline{v_{i+1}}}+\overline{q}+(\overline{v_{k}}+\cdots+\overline{v_{i-1}})(\overline{p}+\overline{q})+(\overline{p}+\overline{q}+\overline{v_{i+1}})\overline{v_i})=0.
\end{align*}
Under this observation,
\begin{align*}
x_{i+1}-t_ix_it_i=&(\frac{1}{2}\kappa_{V}+\gamma_{M,i+1}+\cdots+\gamma_{i,i+1})-\\
&t_i(\frac{1}{2}\kappa_V+\gamma_{M,i}+\cdots+\gamma_{i-1,i})t_i\\
=&\kappa_{i+1}-t_i\kappa_it_i+\gamma_{i,i+1}.
\end{align*}
and $\kappa_{i+1}=t_i\kappa_it_i$ because the action of $k_i$ introduces no signs. Hence we've shown $x_{i+1}-t_ix_it_i=\gamma_{i,i+1}$. Similarly, $y_{i+1}-t_iy_{i+1}t_i=\gamma_{i,i+1}$ and the relation $x_{i+1}-t_ix_it_i=y_{i+1}-t_iy_{i+1}t_i$ holds.\\

Now let us verify (R6).

To see $z_i=x_i+y_i-m_i$ holds, first verify that $m_{i,j}=\gamma_{i,j}$. The calculation is similar to that of $t_{i}\gamma_{k,i}t_i=\gamma_{k,i+1}$. By definition $m_{i,j}=t_{i,j-1}\gamma_{j-1,j}t_{i,j-1}$ where $t_{i,j-1}$ is the signed permutation that interchanges the $i$ and $j-1$ entries, and is generated by the signed swaps. 

\begin{align*}
&m_{i,j}\underline{v}=t_{i,j-1}\gamma_{j-1,j}t_{i,j-1}\underline{v}\\
=&t_{i,j-1}\gamma_{j-1,j}(-1)^{(\overline{v_i}+\overline{v_{j-1}})(\overline{v_{i+1}}+\cdots+\overline{v_{j-2}})+\overline{v_i}\cdot\overline{v_{j-1}}}\\
&\cdot v_M\otimes \cdots\otimes v_{j-1}\otimes \cdots \otimes v_i \otimes v_j\otimes \cdots\otimes v_d\\
=&(-1)^{(\overline{v_i}+\overline{v_{j-1}})(\overline{v_{i+1}}+\cdots+\overline{v_{j-2}})+\overline{v_i}\cdot\overline{v_{j-1}}}t_{i,j-1}\displaystyle\sum_{p,q}(-1)^{\overline{q}}(-1)^{(\overline{p}+\overline{q})\overline{v_i}}\\
&\cdot v_M\otimes \cdots\otimes v_{j-1}\otimes \cdots  \otimes E_{pq}v_i \otimes E_{qp}v_j\otimes \cdots\otimes v_d\\
=&(-1)^{(\overline{v_i}+\overline{v_{j-1}})(\overline{v_{i+1}}+\cdots+\overline{v_{j-2}})+\overline{v_i}\cdot\overline{v_{j-1}}}t_{i,j-1}\displaystyle\sum_{p,q}(-1)^{\overline{q}}(-1)^{(\overline{p}+\overline{q})\overline{v_i}}\\
&(-1)^{(\overline{v_i}+\overline{v_{j-1}}+\overline{p}+\overline{q})(\overline{v_{i+1}}+\cdots+\overline{v_{j-2}})+(\overline{v_i}+\overline{p}+\overline{q})\cdot\overline{v_{j-1}}}\\
&\cdot v_M\otimes \cdots\otimes  E_{pq}v_i\otimes \cdots  \otimes v_{j-1} \otimes E_{qp}v_j\otimes \cdots\otimes v_d\\
=&(-1)^{(\overline{p}+\overline{q})(\overline{v_i}+\cdots+\overline{v_{j-1}})}\displaystyle\sum_{p,q}(-1)^{\overline{q}}\\
&\cdot v_M\otimes \cdots\otimes  E_{pq}v_i\otimes \cdots  \otimes v_{j-1} \otimes E_{qp}v_j\otimes \cdots\otimes v_d\\
=&\gamma_{i,j}\underline{v}.
\end{align*} 

Hence $m_j=\gamma_{1,j}+\cdots+\gamma_{j-1,j}$ for $1<j<d$, and 
\begin{align*}
x_j&=\frac{1}{2}\kappa_V+\gamma_{M,j}+\gamma_{1,j}+\cdots+\gamma_{j-1,j}\\
y_j&=\frac{1}{2}\kappa_V+\gamma_{N,j}+\gamma_{1,j}+\cdots+\gamma_{j-1,j}\\
z_j&=\kappa_V+\gamma_{M,j}+\gamma_{N,j}+\gamma_{1,j}+\dots+\gamma_{j-1,j}.
\end{align*}
Hence $z_j=x_j+y_j-m_j$.

The relation $t_it_{i+1}(x_{i+1}-t_ix_it_i)t_{i+1}t_i=(x_{i+2}-t_{i+1}x_{i+1}t_{i+1})$ now is equivalent to $t_it_{i+1}\gamma_{i,i+1}t_{i+1}t_i=\gamma_{i+1,i+2}$. For short let $v_{begin}=v_M\otimes v_N\otimes \cdots \otimes v_{i-1}$, $v_{end}=v_{i+3}\otimes \cdots\otimes v_d$ (this has a slightly different index from the similar one used earlier.) then

\begin{align*}
&t_it_{i+1}\gamma_{i,i+1}t_{i+1}t_i \underline{v}\\
=&t_it_{i+1}\gamma_{i,i+1}t_{i+1} (-1)^{\overline{v_i}\cdot \overline{v_{i+1}}} v_{begin}\otimes v_{i+1} \otimes v_{i} \otimes v_{i+2} \otimes v_{end}\\
=&t_it_{i+1}\gamma_{i,i+1} (-1)^{\overline{v_i}\cdot \overline{v_{i+1}}+\overline{v_{i}}\cdot \overline{v_{i+2}}} v_{begin}\otimes v_{i+1} \otimes v_{i+2}\otimes v_{i}  \otimes v_{end}\\
=&t_it_{i+1} \displaystyle\sum_{p,q}(-1)^{\overline{q}}(-1)^{\overline{v_i}\cdot \overline{v_{i+1}}+\overline{v_{i}}\cdot \overline{v_{i+2}}+(\overline{p}+\overline{q})\overline{v_{i+1}}}\\
& v_{begin}\otimes E_{pq}v_{i+1} \otimes E_{qp}v_{i+2}\otimes v_{i}  \otimes v_{end}\\
=& t_i\displaystyle\sum_{p,q}(-1)^{\overline{q}}(-1)^{\overline{v_i}\cdot \overline{v_{i+1}}+\overline{v_{i}}\cdot \overline{v_{i+2}}+(\overline{p}+\overline{q})\overline{v_{i+1}}+\overline{v_i}(\overline{p}+\overline{q}+\overline{v_{i+2}})}\\
& v_{begin}\otimes E_{pq}v_{i+1} \otimes  v_{i}  \otimes E_{qp}v_{i+2}\otimes v_{end}\\
=& \displaystyle\sum_{p,q}(-1)^{\overline{q}}(-1)^{\overline{v_i}\cdot \overline{v_{i+1}}+\overline{v_{i}}\cdot \overline{v_{i+2}}+(\overline{p}+\overline{q})\overline{v_{i+1}}+\overline{v_i}(\overline{p}+\overline{q}+\overline{v_{i+2}})+\overline{v_i}(\overline{p}+\overline{q}+\overline{v_{i+1}})}\\
& v_{begin}\otimes   v_{i}  \otimes E_{pq}v_{i+1} \otimes E_{qp}v_{i+2}\otimes v_{end}.
\end{align*}
On the other hand,
\begin{align*}
\gamma_{i+1,i+2}\underline{v}=\displaystyle\sum_{p,q}(-1)^{\overline{q}+(\overline{p}+\overline{q})(\overline{v_{i+1}})}v_{begin}\otimes v_i\otimes E_{pq}v_{i+1}\otimes E_{qp}v_{i+2}\otimes v_{end}
\end{align*}
It remains to check the signs match:
\begin{align*}
&(\overline{q}+\overline{v_i}\cdot \overline{v_{i+1}}+\overline{v_{i}}\cdot \overline{v_{i+2}}+(\overline{p}+\overline{q})\overline{v_{i+1}}+\overline{v_i}(\overline{p}+\overline{q}+\overline{v_{i+2}})+\overline{v_i}(\overline{p}+\overline{q}+\overline{v_{i+1}}))-\\
&(\overline{q}+(\overline{p}+\overline{q})(\overline{v_{i+1}}))=0
\end{align*}
Hence all relations are preserved.
\end{proof}

This allows for another action defined as follows:

\begin{lemma}\label{action}
There is an algebra homomorphism
\begin{align*}
\rho':\mathcal{G}_d &\to \operatorname{End}_{\mathfrak{gl}(n|m)}(M\otimes N\otimes V^{\otimes d})\\
x_i  &\mapsto \rho(x_i)-\frac{1}{2}\kappa_V \hspace{.2 in} 1\leq i\leq d\\
y_i  &\mapsto \rho(y_i)-\frac{1}{2}\kappa_V \hspace{.2 in} 1\leq i\leq d \\
z_i  &\mapsto \rho(z_i)-\kappa_V \hspace{.2 in} 1\leq i \leq d \\
z_0 & \mapsto \rho(z_0).
\end{align*}
\end{lemma}
\begin{proof}
The constants satisfy the conditions in  \cite[Lemma~3.4]{D} and therefore defines an algebra automorphism $\phi: \mathcal{G}_d \to \mathcal{G}_d$ where generators $x_i$, $y_i$, $z_i$ are shifted by the given constants. The homomorphism $\rho'=\rho \circ \phi$ is as above and therefore an algebra homomorphism.
\end{proof}

\subsection{The Hecke algebra $\mathcal{H}^{\operatorname{ext}}_d$ and its action}

Recall the set of hook Young diagrams satisfying condition (H1) and the set of weights satisfying condition (H2) established in Section~\ref{young} and the bijection between them through the map $\overline{\cdot}$. For the remainder of this article, fix $a,b,p,q \in \mathbb{Z}_{>0}$ such that $a,b \leq m$, $a\geq p-n$ and $b\geq q-n$. Denote by $(a^p)$ the partition $a\geq a \geq \cdots \geq a \geq 0 \geq \cdots$ with $p$ copies of $a$ (i.e. the rectangle with $p$ rows of $a$ boxes,) $(b^q)$ the partition $b\geq b \geq \cdots \geq b \geq 0 \geq \cdots  $ with $q$ copies of $b$ (i.e. the rectangle with $q$ rows of $b$ boxes.) Note the conditions on $a,b,p,q$ ensures that $(a^p)$ and $(b^q)$ are hook Young diagrams. Define the following weights

\begin{align*}
\alpha=\overline{(a^p)}=\begin{cases}
(a^p) \hspace{.1 in} \text{if }p\leq n \\
(a^n)\cup ((p-n)^a) \hspace{.1 in}\text{if } p>n
\end{cases},\\
\beta=\overline{(b^q)}=\begin{cases}
(b^q) \hspace{.1 in} \text{if }q\leq n \\
(b^q)\cup ((q-n)^b) \hspace{.1 in} \text{if } q>n
\end{cases}.
\end{align*}

Here, by $(a^n)\cup ((p-n)^a)$ we mean putting the rectangle $((p-n)^a)$ right underneath the rectangle $(a^n)$, and similarly for the partition $(b^q)\cup ((q-n)^b) $. 

For the given choices of positive integers $a,b,p,q$ earlier, let the extended degenerate two-boundary Hecke algebra $\mathcal{H}^{\operatorname{ext}}_d$ be the quotient of $\mathcal{G}_d$ under the added relations:
\begin{align*}
(x_1-a)(x_1+p)&=0\\
(y_1-b)(y_1+q)&=0\\
x_{i+1}&=t_ix_it_i+t_i\\
y_{i+1}&=t_iy_it_i+t_i\\
&(1\leq i\leq d-1)
\end{align*}

Recall that $\alpha=\overline{(a^p)}$, $\beta=\overline{(b^q)}$. From now on take $M=L(\alpha)$ and $N=L(\beta)$ for the rest of the discussion. The goal of this section is to show that the action of $\mathcal{G}_d$ factors through the relations and induces an action of $\mathcal{H}^{\operatorname{ext}}_d$ on  $L(\alpha)\otimes L(\beta)\otimes V^{\otimes d}$. This requires a few lemmas.

\begin{lemma}\label{weightaction}
Let $L(\lambda)$ be defined as before, then $\kappa$ acts on $L(\lambda)$ by a scalar 
\begin{align*}
\langle\lambda,\lambda+2\rho\rangle.
\end{align*}
\end{lemma}

\begin{proof}
Since $L(\lambda)$ is simple, we can calculate the scalar action of $\kappa$ on the highest weight vector $v_{\lambda}$ in $L(\lambda)$. Without further notation, the following sum is assumed over all integers $1\leq i,j\leq n+m$. In particular,
\begin{align*}
\kappa.v_{\lambda}&=\displaystyle\sum_{i,j}(-1)^{\bar{j}}E_{i,j}E_{j,i}v_{\lambda}\notag\\
=&\displaystyle\sum_{i}E_{i,i}^2v_{\lambda}+\displaystyle\sum_{i<j}(-1)^{\bar{j}}E_{i,j}E_{j,i}v_{\lambda}+\displaystyle\sum_{i>j}(-1)^{\bar{j}}E_{i,j}E_{j,i}v_{\lambda} \notag \\
=&\displaystyle\sum_{i}E_{i,i}^2v_{\lambda}+\displaystyle\sum_{i<j}(-1)^{\bar{j}}(E_{ii}-(-1)^{(\bar{i}+\bar{j})^2}E_{jj}+(-1)^{(\bar{i}+\bar{j})^2}E_{ji}E_{ij})v_{\lambda}\\
&+\displaystyle\sum_{i>j}(-1)^{\bar{i}}E_{i,j}E_{j,i}v_{\lambda} \notag \\
=&\displaystyle\sum_{i}h_{i}^2v_{\lambda}-\displaystyle\sum_{i<j}(-1)^{\bar{i}}h_jv_{\lambda}+\displaystyle\sum_{i<j}(-1)^{\bar{j}}h_iv_{\lambda}.
\end{align*}
Here by $i<j$ we mean the sum is over all paris $(i,j)$ such that $i<j$. We also used the fact that 
\begin{align*}
[E_{ji},E_{ij}]=E_{jj}-(-1)^{\bar{i}+\bar{j}}E_{ii},
\end{align*}
where $\bar{i}+\bar{j}=(\bar{i}+\bar{j})^2$, and notice $E_{ij}v_{\lambda}=0$ for $i<j$.
For the bilinear form $\langle\epsilon_i,\epsilon_j\rangle=(-1)^{\bar{i}}\delta_{ij}$ defined on $\mathfrak{h}^*$, 
\begin{align*}
h_iv_{\lambda}=\lambda(h_i)v_{\lambda}=\lambda_iv_{\lambda}=(-1)^{\bar{i}}\langle\lambda,\epsilon_i\rangle v_{\lambda}.
\end{align*}
The above calculation becomes
\begin{align*}
\kappa v_{\lambda}&= \displaystyle\sum_{i}\lambda_i^2v_{\lambda}-\displaystyle\sum_{i<j}(-1)^{\bar{i}+\bar{j}}\langle \lambda,\epsilon_j\rangle v_{\lambda}+\displaystyle\sum_{i<j}(-1)^{\bar{i}+\bar{j}} \langle \lambda,\epsilon_i \rangle v_{\lambda} \\
&=\langle \lambda,\lambda \rangle v_{\lambda}+\displaystyle\sum_{i<j}(-1)^{\bar{i}+\bar{j}}\langle \lambda,\epsilon_i-\epsilon_j\rangle v_{\lambda}\\
&=\langle \lambda,\lambda+2\rho\rangle v_{\lambda},
\end{align*}
since $\bar{i}+\bar{j}=0$ precisely when $\epsilon_i-\epsilon_j$ is even.
\end{proof}

We also need the following lemma
\begin{lemma}\label{inner}
The following is true
\begin{align*}
\langle\epsilon_i,\epsilon_i+2\rho \rangle=\begin{cases}
-2i+2+n-m \hspace{.5 in}\text{if }1\leq i\leq n \\
2i-2-3n-m \hspace{.5 in}\text{if }n+1\leq i\leq m
\end{cases}.
\end{align*}
\end{lemma}
\begin{proof}
This is a straightforward calculation.
\end{proof}

Note that for the natural representation $V$ of $\mathfrak{g}$, $V=L(\epsilon_1)$. 

\begin{lemma}\label{pierri}
When $L(\lambda)$ occurs as a direct summand of $L(\mu)\otimes V$, $\gamma$ acts on $L(\lambda)$ as a scalar 
\begin{align*}
\gamma^{\lambda}_{\mu,\epsilon_1}=c(b) 
\end{align*}

where $c(b)$ is the content of the box $b$ added to $\overline{\mu}$ to obtain $\overline{\lambda}$.
\end{lemma}

\begin{proof}
On one hand, $\kappa$ acts on $L(\lambda)$ as the scalar specified in Lemma \ref{weightaction}. On the other hand, since
\begin{align*}
\Delta(\kappa)=\kappa \otimes 1 +1\otimes \kappa +2\gamma,
\end{align*}
Lemma \ref{weightaction} implies that $\kappa \otimes 1$ and $1 \otimes \kappa$ act on $L(\mu)\otimes L(\epsilon_1)$ by $\langle\mu,\mu+2\rho\rangle$ and $ \langle\epsilon_1,\epsilon_1+2\rho \rangle$,  hence $\gamma$ acts as a scalar 
\begin{align*}
2\gamma_{\mu,\epsilon_1}^{\lambda}=\langle\lambda,\lambda+2\rho\rangle-\langle\mu,\mu+2\rho\rangle-\langle\epsilon_1,\epsilon_1+2\rho\rangle.
\end{align*}
This depends on whether the added box is in the even rows or odd rows. Let's discuss by cases:
1) The box $b$ is added to the $l$-th column and $r$-th row to obtain $\overline{\lambda}$, then 
\begin{align*}
\lambda&=\mu+\epsilon_r\\
\mu_r&=l-1
\end{align*}
\begin{align*}
2\gamma_{\mu,\epsilon_1}^{\lambda}&=\langle\mu+\epsilon_r,\mu+\epsilon_r+2\rho\rangle-\langle\mu,\mu+2\rho\rangle-\langle\epsilon_1,\epsilon_1+2\rho\rangle\notag\\
&=\langle\epsilon_r,\epsilon_r+2\rho\rangle-\langle\epsilon_1,\epsilon_1+2\rho\rangle+2\langle\mu,\epsilon_r\rangle\\
&=(-2r+2+n-m)-(n-m)+2(l-1)\\
&=2(l-r)=2c(b).
\end{align*}

2) The added box is in the rows $n+1$ and below of $\overline{\lambda}$.
Let the added box be in the $l$-th column and $r$-th row of $\overline{\lambda}$. In the subdiagram cut out from the odd part, it is in the $(r-n)$-th row and $l$-th column, and after transposing, it becomes the $(r-n)$-th column and $(n+l)$-th row.  Hence
\begin{align*}
\lambda&=\mu+\epsilon_{n+l}\\
\mu_{n+l}&=r-n-1
\end{align*}
\begin{align*}
\gamma^{\lambda}_{\mu,\epsilon_1}&=\langle\mu+\epsilon_{n+l},\mu+\epsilon_{n+l}+2\rho\rangle-\langle\mu,\mu+2\rho\rangle-\langle\epsilon_1,\epsilon_1+2\rho\rangle\notag\\
&=\langle\epsilon_{n+l},\epsilon_{n+l}+2\rho\rangle-\langle\epsilon_1,\epsilon_1+2\rho\rangle+2\langle\mu,\epsilon_{n+l}\rangle\\
&=(2(n+l)-2-3n-m)-(n-m)-2(r-n-1)\\
&=2(l-r)=2c(b).
\end{align*}
\end{proof}

\begin{theorem}
Let $\rho'$ be the action mentioned in Lemma \ref{action}. Then the extra defining relations for $\mathcal{H}^{\operatorname{ext}}_d$ hold in the image of $\rho'$, hence inducing an action 
\begin{align*}
\Psi: \mathcal{H}^{\operatorname{ext}}_d \to \operatorname{End}_{\mathfrak{g}}(L(\alpha)\otimes L(\beta)\otimes V^{\otimes d}).
\end{align*}
\end{theorem}
\begin{proof}
For short let us write $M=L(\alpha)$ and $N=L(\beta)$. The action of $x_1$ is given by
\begin{align*}
\frac{1}{2}(\kappa_{M,1}-\kappa_{M}-\kappa_V)=\gamma_{M,1}
\end{align*}
and by the previous lemma, $\gamma_{M,1}$ acts by the scalar $c(b)$, where $b$ is the box added to $(a^p)$ to obtain the next diagram. Since this box can only be added to the end of the first row or the bottom of the first column, the content can only be $a$ or $-p$, therefore 
\begin{align*}
(x_1-a)(x_1+p)=0.
\end{align*}
Similarly, 
\begin{align*}
(y_1-b)(y_1+q)=0.
\end{align*}
In the proof of Theorem \ref{action}, we showed that $x_{i+1}-t_ix_it_i=\gamma_{i,i+1}$. Hence in order to check the further relation $x_{i+1}-t_ix_it_i=t_i$ in $\mathcal{H}^{\operatorname{ext}}$, it is enough to check the action of $t_i$ agrees with that of $\gamma_{i,i+1}$. According to the Littlewood-Richardson rule, the decomposition of $V\otimes V$, as the $i$-th and $i+1$-th copy in $M\otimes N\otimes V^{\otimes d}$, is the following
\begin{align*}
L(\begin{ytableau}*(white)\end{ytableau})\otimes L(\begin{ytableau}*(white)\end{ytableau}) \simeq L(\begin{ytableau}*(white)& *(white)\end{ytableau})\oplus L(\begin{ytableau}*(white)\\*(white)\end{ytableau})
\end{align*}
where $\gamma_{i,i+1}$ acts on $L(\begin{ytableau}*(white)& *(white)\end{ytableau})$ by $1$ and acts on $L(\begin{ytableau}*(white)\\*(white)\end{ytableau})$ by $-1$ based on the content of the added box. On the other hand,
\begin{align*}
L(\begin{ytableau}*(white)& *(white)\end{ytableau})&=\mathbb{C}-\operatorname{span}\{v_i\otimes v_j +(-1)^{\overline{i}\cdot \overline{j}}v_j\otimes v_i|1\leq i,j\leq n+m\}\\
L(\begin{ytableau}*(white)\\*(white)\end{ytableau})&=\mathbb{C}-\operatorname{span}\{v_i\otimes v_j -(-1)^{\overline{i}\cdot \overline{j}}v_j\otimes v_i |1\leq i,j\leq n+m \}
\end{align*}
and $s_i$ also acts via eigenvalues $1$ and $-1$, therefore agrees with the action of $\gamma_{i,i+1}$.
\end{proof}

\section{Seminormal Representations}
\subsection{The Bratteli graph and double centralizer theorem}\label{bra}
Let the integers $n,m,a,b,p,q$ be chosen as earlier. Let the associated Bratteli graph $\Gamma$ be the directed graph with vertices $\mathcal{P}=\displaystyle\bigcup_{i=-1}^{\infty}\mathcal{P}_i$. Here, each $\mathcal{P}_i$ is a set of hook Young diagrams defined as follows: $\mathcal{P}_{-1}=(a^p)$, and 
\begin{align*}
\mathcal{P}_0&=\{\lambda\hspace{.1 in}|\hspace{.1 in} L(\overline{\lambda}) \text{ is a summand of }L(\alpha)\otimes L(\beta)\}\\
\mathcal{P}_i&=\{\lambda\hspace{.1 in}|\hspace{.1 in} L(\overline{\lambda}) \text{ is a summand of }L(\alpha)\otimes V\} \hspace{.2 in} i\geq 1
\end{align*}
% or equivalently, the set of all $\lambda$'s with $i$ more boxes than $(a^p)$ such that the Littlewood-Richardson coefficient $c^{\lambda}_{\mu,\epsilon_1}\neq 0$ for some $\mu\in \mathcal{P}_{i-1}$.
 Recall from the discussion in Section \ref{young} and \ref{lr} that this implies for $i\geq 1$, $\lambda\in \mathcal{P}_i$ if and only if $\lambda$ is a hook partition and can be obtained by adding a box to some diagram in $\mathcal{P}_{i-1}$. Similarly, there are combinatorial rules determining the set of hood partitions lying in $\mathcal{P}_0$. We will refer to partitions in $\mathcal{P}_i$ as vertices at level $i$.

The edges in the graph are the following: there is a directed edge from $\lambda\in \mathcal{P}_i$ to $\mu \in \mathcal{P}_{i+1}$ if $\lambda$ is contained in $\mu$ for $i\geq 0$. The only vertex in rank $-1$ has edges pointing to all vertices in rank $0$.

In the case $(a^p)=(4^3)$, $(b^q)=(2^2)$, $n=3$, $m=1$, the first few rows of the Bratteli diagram are as the following. The even and odd parts of each partition are indicated using different colors.

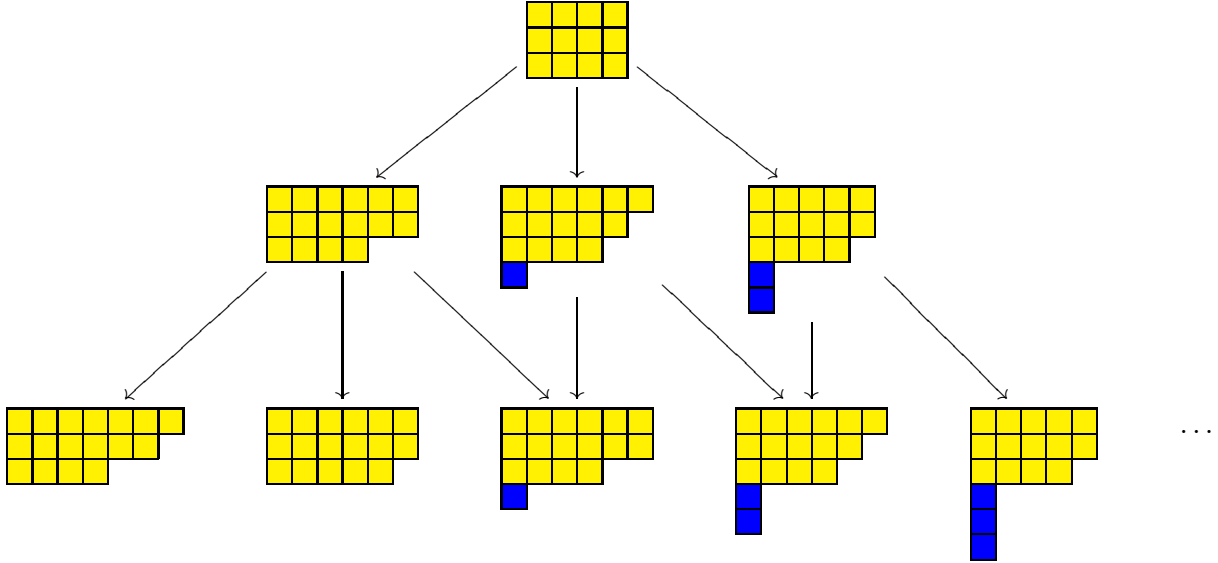
\xymatrix{
&&{\begin{ytableau}
*(yellow)&*(yellow)&*(yellow)&*(yellow)\\
*(yellow)&*(yellow)&*(yellow)&*(yellow)\\
*(yellow)&*(yellow)&*(yellow)&*(yellow)
\end{ytableau}} \ar[dl]\ar[d]\ar[dr]&&\\
&{\begin{ytableau}
*(yellow)&*(yellow)&*(yellow)&*(yellow)&*(yellow)&*(yellow)\\
*(yellow)&*(yellow)&*(yellow)&*(yellow)&*(yellow)&*(yellow)\\
*(yellow)&*(yellow)&*(yellow)&*(yellow)
\end{ytableau}}\ar[dl]\ar[d]\ar[dr] 
&{\begin{ytableau}
*(yellow)&*(yellow)&*(yellow)&*(yellow)&*(yellow)&*(yellow)\\
*(yellow)&*(yellow)&*(yellow)&*(yellow)&*(yellow)\\
*(yellow)&*(yellow)&*(yellow)&*(yellow)\\
*(blue)
\end{ytableau}} \ar[d]\ar[dr]
& {\begin{ytableau}
*(yellow)&*(yellow)&*(yellow)&*(yellow)&*(yellow)\\
*(yellow)&*(yellow)&*(yellow)&*(yellow)&*(yellow)\\
*(yellow)&*(yellow)&*(yellow)&*(yellow)\\
*(blue)\\
*(blue)
\end{ytableau}}\ar[d]\ar[dr]&&\\
{\begin{ytableau}
*(yellow)&*(yellow)&*(yellow)&*(yellow)&*(yellow)&*(yellow)&*(yellow)\\
*(yellow)&*(yellow)&*(yellow)&*(yellow)&*(yellow)&*(yellow)\\
*(yellow)&*(yellow)&*(yellow)&*(yellow)
\end{ytableau}} 
&{\begin{ytableau}
*(yellow)&*(yellow)&*(yellow)&*(yellow)&*(yellow)&*(yellow)\\
*(yellow)&*(yellow)&*(yellow)&*(yellow)&*(yellow)&*(yellow)\\
*(yellow)&*(yellow)&*(yellow)&*(yellow)&*(yellow)
\end{ytableau}}
& {\begin{ytableau}
*(yellow)&*(yellow)&*(yellow)&*(yellow)&*(yellow)&*(yellow)\\
*(yellow)&*(yellow)&*(yellow)&*(yellow)&*(yellow)&*(yellow)\\
*(yellow)&*(yellow)&*(yellow)&*(yellow)\\
*(blue)
\end{ytableau}}
&{\begin{ytableau}
*(yellow)&*(yellow)&*(yellow)&*(yellow)&*(yellow)&*(yellow)\\
*(yellow)&*(yellow)&*(yellow)&*(yellow)&*(yellow)\\
*(yellow)&*(yellow)&*(yellow)&*(yellow)\\
*(blue)\\
*(blue)
\end{ytableau}}
 & {\begin{ytableau}
*(yellow)&*(yellow)&*(yellow)&*(yellow)&*(yellow)\\
*(yellow)&*(yellow)&*(yellow)&*(yellow)&*(yellow)\\
*(yellow)&*(yellow)&*(yellow)&*(yellow)\\
*(blue)\\
*(blue)\\
*(blue)
\end{ytableau}} &\dots
}

Notice the following diagram 
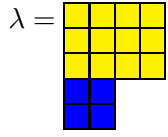
\begin{align*}
\lambda={\begin{ytableau}
*(yellow)&*(yellow)&*(yellow)&*(yellow)\\
*(yellow)&*(yellow)&*(yellow)&*(yellow)\\
*(yellow)&*(yellow)&*(yellow)&*(yellow)\\
*(blue)&*(blue)\\
*(blue)&*(blue)
\end{ytableau}}
\end{align*}
does not appear in the second row of the Bratteli graph, even though the (nonhook) Littlewood-Richardson coefficient $c^{\lambda}_{\alpha,\beta}$ is nonzero, because the bottom two rows lie outside the $(3,1)$-hook. 

%These combinatorial objects are known to control the representation theory of the Lie algebra and its centralizer. The following result is well-known
 
%\begin{theorem}{(Double Centralizer Theorem)}
%Let $U$ be a finite dimensional $\mathbb{C}$-vector space. Let $A \subset \operatorname{End}(U)$ be a subalgebra of the endomorphism ring. Let $B=\operatorname{End}_A(U)$ be the centralizer, then $A=\operatorname{End}_B(U)$. Moreover, $U$ decomposes as $A$, $B$-bimodules
%\begin{align*}
%U\simeq \displaystyle\bigoplus_{\lambda}L_{\lambda} \otimes \mathcal{L}^{\lambda}
%\end{align*}
%where the direct sum is over some finite index set, $L_{\lambda}$'s are a list of pairwise nonisomorphic irreducible $A$-modules, $\mathcal{L}^{\lambda}$'s are a list of pairwise nonisomorphic irreducible $B$-modules.
%\end{theorem}

We now discuss summands of $L(\alpha)\otimes L(\beta) \otimes V^{\otimes d}$ as a module for $\mathcal{H}^{\operatorname{ext}}_d$. Define the centralizer algebra $\mathcal{H}=\operatorname{End}_{\mathfrak{gl}(n|m)}(L(\alpha)\otimes L(\beta)\otimes V^{\otimes})$. According to the double centralizer theorem,
\begin{align*}
L(\alpha)\otimes L(\beta) \otimes V^{\otimes d} = \displaystyle \bigoplus_{\lambda\in \mathcal{P}_d} L(\lambda) \otimes \mathcal{L}^{\lambda}
\end{align*}
as $(\mathfrak{gl}(n|m),\mathcal{H})$-bimodules. Moreover, the dimension of $\mathcal{L}^{\lambda}$ as a $\mathbb{C}$-vector space, is equal to the multiplicity of $L(\lambda)$ as a $\mathfrak{gl}(n|m)$-module in the decomposition.

Notice that the action of $\mathfrak{h}$ commutes with $\mathcal{H}$, therefore
\begin{align*}
(L(\lambda) \otimes \mathcal{L}^{\lambda})_{\lambda}&=\{v\in L(\lambda) \otimes \mathcal{L}^{\lambda}\hspace{.1in} |\hspace{.1in} h.v=\lambda(h)v, \forall h\in \mathfrak{h}\}\\
&=(L(\lambda))_{\lambda}\otimes  \mathcal{L}^{\lambda}
\end{align*}
is an $\mathcal{H}$-module. On the other hand, $(L(\lambda))_{\lambda}$ is known to be one-dimensional, spanned by the highest weight vector $v_{\lambda}$ in $L(\lambda)$. Therefore $(L(\lambda) \otimes \mathcal{L}^{\lambda})_{\lambda}$ and $\mathcal{L}^{\lambda}$ have the same dimension, and $\mathcal{L}^{\lambda}$ consists of highest weight vectors from the various copies of $L(\lambda)$ in the decomposition.

We label each highest weight vector $v_{\lambda}$ via the choice of the specific copy of $L(\lambda)$ using the Bratteli graph. For a fixed choice of $d$ and $\lambda\in \mathcal{P}_d$, let $\Gamma^{\lambda}$ be the set of directed paths from $\alpha$ to $\lambda$. For each given path $T\in \Gamma^{\lambda}$, let
\begin{align*}
T=(\alpha,T^{(0)},\dots,T^{(d)}=\lambda)
\end{align*}
be the vertices along the path. We now slightly abuse the notation $\lambda$ for both the partition and the weight, and use $L(\lambda)$ to denote the highest weight $\mathfrak{gl}(n|m)$-module associated to the highest weight $\overline{\lambda}$.

Given a path $T$, one can obtain a copy of $L(\lambda)$ uniquely, via the following:  since the decomposition of $L(\alpha)\otimes L(\beta)$ is multiplicity free, one can choose the unique summand $L({T^{(0)}})$. For $i\geq 1$, choose the summand inductively by choosing $L({{T^{(i)}}})$ uniquely in the decomposition of $L({{T^{(i-1)}}})\otimes V$. Therefore, the highest weight vector $v_{\lambda}$ (up to a choice of rescaling) in a copy $L(\lambda)$ is a linear combination of vectors of the form $w_d\otimes u_d$, with $w_d\in L(T^{(d-1)})$ and $u_d\in V$. Again each $w_d$ is a linear combination of vectors of the form $w_{d-1}\otimes u_{d-1}$, with $w_{d-1}\in L(T^{(d-2)})$ and $u_{d-1}\in V$, etc. Denote this highest weight vector by $v_T$, then $\mathcal{L}^{\lambda}$ admits a basis $\{v_T\}_{T\in \Gamma^{\lambda}}$.

Recall that the algebra $\mathcal{H}^{\operatorname{ext}}_d$ is generated by $z_0,\dots,z_d$, $x_1$, $s_1,\dots,s_{d-1}$.

\begin{theorem}\label{z0}
Fix $\lambda\in \mathcal{P}_d$, let $v_T$ be as defined as above, where $T\in \Gamma^{\lambda}$ is a path in the Bratteli graph $\Gamma$, then
\begin{align*}
z_i.v_T=c(T^{(i)}/T^{(i-1)})v_T.
\end{align*}
Here, $1\leq i\leq d$, $c(T^{(i)}/T^{(i-1)})$ denotes the content of the box added to $T^{(i-1)}$ in order to obtain $T^{(i)}$.
\end{theorem}

\begin{proof}
Let us write $M=L(\alpha)$ and $N=L(\beta)$ for short. In particular, for an element $v_T$, by the discussion mentioned earlier, $v_T$ is a linear combination of vectors of the form $v_{T^{(i)}}\otimes w$, with $v_{T^{(i)}}\in L(T^{(i)})$ and $w\in V^{\otimes d-i}$, and $v_{T^{(i)}}$ is again a linear combination of vectors of the form $v_{T^{(i-1)}}\otimes u$ with $v_{T^{(i-1)}}\in L^{(i-1)}$ and $u\in V$.
\begin{align*}
z_i.v_T^{}&=\frac{1}{2}(\kappa_{M\otimes N, i}-\kappa_{M\otimes N, i-1}-\kappa_V).v_T^{}\\
&=\frac{1}{2}\sum \kappa_{M\otimes N, i}.v_{T^{(i)}}\otimes w -\frac{1}{2}\sum \kappa_{M\otimes N, i-1}v_{T^{(i-1)}}\otimes u - \frac{1}{2}\langle\epsilon_1,\epsilon_1+2\rho \rangle v_T^{},
\end{align*}
where the sum is over all the terms in $v_T$. Since the action of $\kappa_V$ on $V$ is $\langle\epsilon_1,\epsilon_1+2\rho \rangle$ according to Lemma \ref{weightaction}. Applying the same lemma to the first two actions on the corresponding vector, the above formula becomes
\begin{align*}
z_i.v_T^{}&=\frac{1}{2}(\langle\lambda', \lambda'+2\rho\rangle-\langle \lambda, \lambda+2\rho\rangle-\langle\epsilon_1,\epsilon_1+2\rho \rangle )v_T^{}
\end{align*}
Where $\lambda$, $\lambda'$ be the weight associated to $T^{(i-1)}$, $T^{(i)}$, i.e. $\overline{\lambda}=T^{(i-1)}$ and $\overline{\lambda'}=T^{(i)}$.

Now we discuss by cases:

1) The diagrams $T^{(i-1)}$ and $T^{(i)}$ differ by a box $b$ in the rows $n+1$ and below in either diagram.
Let the added box be in the $l$-th column and $r$-th row of $T^{(i)}$. In the subdiagram cut out from the odd part, it is in the $(r-n)$-th row, and after transposing, it becomes the $(r-n)$-th column and $(n+l)$-th row.  Hence
\begin{align*}
\lambda'&=\lambda+\epsilon_{n+l}\\
\lambda_{n+l}&=r-n-1
\end{align*}
\begin{align*}
&\langle\lambda', \lambda'+2\rho\rangle-\langle \lambda, \lambda+2\rho\rangle-\langle\epsilon_1,\epsilon_1+2\rho \rangle \\
=&\langle\lambda+\epsilon_{n+l}, \lambda+\epsilon_{n+l}+2\rho\rangle-\langle \lambda, \lambda+2\rho\rangle-\langle\epsilon_1,\epsilon_1+2\rho \rangle \\
=&2 \langle \lambda,\epsilon_{n+l} \rangle + \langle \epsilon_{n+l}, \epsilon_{n+l}+2\rho \rangle - \langle \epsilon_1, \epsilon_1+2\rho \rangle \\
=&-2 (r-n-1) + (2(n+l)-2-3n-m)-(n-m) \text{   (Lemma \ref{inner})}\\
=&-2r+2l=2c(b).
\end{align*}

2) The diagrams $T^{(i-1)}$ and $T^{(i)}$ differ by a box in the rows $n$ and above in either diagram.
This case is similar to the $\mathfrak{gl}_n$ calculation but we'll repeat it here. Let $b$ be the box added to $T^{(i-1)}$ to obtain $T^{(i)}$, and $b$ is in the $l$-th column and $r$-th row of $T^{(i)}$. Since $b$ stays in the even part in the process of $\lambda \mapsto \overline{\lambda}$, it follows that 
\begin{align*}
\lambda'&=\lambda+\epsilon_r\\
\lambda_r&=l-1
\end{align*}
 \begin{align*}
&\langle\lambda', \lambda'+2\rho\rangle-\langle \lambda, \lambda+2\rho\rangle-\langle\epsilon_1,\epsilon_1+2\rho \rangle \\
=&\langle\lambda+\epsilon_{r}, \lambda+\epsilon_{r}+2\rho\rangle-\langle \lambda, \lambda+2\rho\rangle-\langle\epsilon_1,\epsilon_1+2\rho \rangle \\
=&2\langle \lambda,\epsilon_r \rangle + \langle \epsilon_r, \epsilon_r +2\rho \rangle - \langle \epsilon_1 , \epsilon_1 +2\rho \rangle\\
=&2(l-1) + (-2r+2+n-m)-(n-m) \\
=&2(l-r)=2c(b).
 \end{align*} 
\end{proof}

\begin{remark}
Similar to the above construction, one can define a Bratteli graph $\Phi$ using the sequence of decomposition of $M\otimes V^{\otimes d}\otimes N$, then the above fixed partition $\lambda$ also appears at level $d$. Let $\Phi^{\lambda}$ be the set of paths from $\alpha$ to $\lambda$, then the $\mathcal{H}$-module $\mathcal{L}^{\lambda}$ admits a basis $\{v_T^{(y)}\}_{T\in \Phi^{\lambda}}$, with
\begin{align*}
y_i.v_T^{(y)}=c(T^{(i)}/T^{(i-1)})v_T^{(y)}
\end{align*}
for $1\leq i\leq d$.

One can also define a Bratteli graph $\Psi$ using the sequence of decomposition of $N\otimes V^{\otimes d}\otimes M$. Let $\Psi^{\lambda}$ be the set of paths from $\alpha$ to $\lambda$, then the $\mathcal{H}$-module $\mathcal{L}^{\lambda}$ admits a basis $\{v_T^{(z)}\}_{T\in \Phi^{\lambda}}$, with
\begin{align*}
z_i.v_T^{(z)}=c(T^{(i)}/T^{(i-1)})v_T^{(z)}
\end{align*}
for $1\leq i\leq d$.
\end{remark}

\subsection{The action of $z_0$}
The previous secion discusses the action of all polynomial generators except for $z_0$. There are some combinatorial facts about diagrams in $\mathcal{P}_0$. In \cite{Ok,Stan}, the authors gave a combinatorial rule for partitions occuring in $\mathcal{P}_0$.
% all partitions $\lambda$ with Littlewood-Richardson coefficients $c^{\lambda}_{(a^p),(b^q)}$ can be obtained by cutting out a starcase corner of $(a^p)+(b^q)$ (i.e. the diagram of combining the two rectangles side by side horizontally,) rotate it by $180$ degrees, and put it underneath the rectangle $(a^p)$.
Based on the above result, Daugherty \cite{D} showed that diagrams in $\mathcal{P}_0$ satisfy the following lemma. We will discuss more on the consequences of the lemma in the next section.
\begin{lemma}{(Lemma 4.13, \cite{D})}\label{P0comb}
When $\lambda$ and $\mu$ are partitions in $\mathcal{P}_0$ that differ by a box, then 
\begin{align*}
c(\text{the distinct box in }\lambda)+c(\text{the distinct box in }\mu)=a-p+b-q
\end{align*}
\end{lemma}

Recall that $\kappa$ is central in $U(\mathfrak{gl}(n|m))$ and acts on any irreducible module $W$ by a scalar. Denote this scalar by $\kappa_W$.
\begin{lemma}\label{transfer}
If $L(\lambda)$ and $L(\mu)$ are irreducible direct summands in $M\otimes N$, where $\overline{\lambda}$ and $\overline{\mu}$ differ by a box, then
\begin{align*}
\kappa_{L(\lambda)}-\kappa_{L(\mu)}=2c(\text{the distinct box in }\overline{\lambda})-2c(\text{the distinct box in }\overline{\mu})
 \end{align*}
Moreover, this value is equal to
\begin{align*}
4(c(b)-\frac{1}{2}(a-p+b-q))
\end{align*}
Where $b$ denotes the distinct box in $\overline{\lambda}$.
\end{lemma}

\begin{proof}
The second part of the claim follows directly from Lemma \ref{P0comb}.

Now let us prove the first part of the claim. It contains a few cases.\\

1) The distinct box $(l,r)$ in $\overline{\mu}$ is in an even row, and the distinct box $(l',r')$ in $\overline{\lambda}$ is in an odd row, where $l$, $l'$ denote the columns and $r$, $r'$ denote the rows of the boxes. Let $\overline{\gamma}$ be the largest diagram that is common in both $\lambda$ and $\mu$. Based on the discussion in Lemma \ref{pierri},
\begin{align*}
\mu&=\gamma+\epsilon_r\\
\gamma_r&=l-1\\
\lambda&=\gamma+\epsilon_{l'+n}\\
\gamma_{l'+n}&=r'-n-1
\end{align*}
\begin{align*}
\kappa_{\lambda}-\kappa_{\mu}&=\langle\lambda,\lambda+2\rho\rangle-\langle\mu,\mu+2\rho\rangle\\
=&\langle\gamma+\epsilon_{l'+n},\gamma+\epsilon_{l'+n}+2\rho\rangle-\langle\gamma+\epsilon_r,\gamma+\epsilon_r+2\rho\rangle\\
=&\langle\epsilon_{l'+n},\epsilon_{l'+n}+2\rho\rangle+2\langle\gamma,\epsilon_{l'+n}\rangle-\langle\epsilon_r,\epsilon_r+2\rho\rangle-2\langle\gamma,\epsilon_r\rangle\\
=&-3n-m+2(l'+n)-2-2(r'-n-1)-\\
&(-2r+2+n-m)-2(l-1)\\
=&2(l'-r')-2(l-r).
\end{align*}

2) The distinct box $(l,r)$ in $\overline{\mu}$ is in an odd row, and the distinct box $(l',r')$ in $\overline{\lambda}$ is in an even row. This case can be seen by switching the roles of $\lambda$ and $\mu$ in the above argument.\\

3) The distinct box $(l,r)$ in $\overline{\mu}$ is in an even row, and the distinct box $(l',r')$ in $\overline{\lambda}$ is in an even row. 
\begin{align*}
\mu&=\gamma+\epsilon_r\\
\gamma_r&=l-1\\
\lambda&=\gamma+\epsilon_{r'}\\
\gamma_{r'}&=l'-1
\end{align*}
\begin{align*}
\kappa_{\lambda}-\kappa_{\mu}&=\langle\epsilon_{r'},\epsilon_{r'}+2\rho\rangle+2\langle\gamma,\epsilon_{r'}\rangle-\langle\epsilon_r,\epsilon_r+2\rho\rangle-2\langle\gamma,\epsilon_r\rangle\\
&=(-2r'+2+n-m)+2(l'-1)-(-2r+2+n-m)-2(l-1)\\
&=2(l'-r')-2(l-r).
\end{align*}

4) The distinct box $(l,r)$ in $\overline{\mu}$ is in an odd row, and the distinct box $(l',r')$ in $\overline{\lambda}$ is in an odd row. 
\begin{align*}
\mu&=\gamma+\epsilon_{l+n}\\
\gamma_{l+n}&=r-n-1\\
\lambda&=\gamma+\epsilon_{l'+n}\\
\gamma_{l'+n}&=r'-n-1.
\end{align*}
\begin{align*}
\kappa_{\lambda}-\kappa_{\mu}=&\langle\epsilon_{l'+n},\epsilon_{l'+n}+2\rho\rangle+2\langle\gamma,\epsilon_{l'+n}\rangle-\\
&\langle\epsilon_{l+n},\epsilon_{l+n}+2\rho\rangle-2\langle\gamma,\epsilon_{l+n}\rangle\\
=&-3n-m+2(l'+n)-2-2(r'-n-1)-\\
&(-3n-m+2(l+n)-2)-2(r-n-1)\\
=&2(l'-r')-2(l-r).
\end{align*}

\end{proof}

\begin{lemma}\label{advinner}
Let $\phi=\epsilon_1+\cdots+\epsilon_t$, where $t\leq n$ be a weight, then 
\begin{align*}
\langle \phi_t,\phi_t+2\rho \rangle =(-t+1+n-m)t
\end{align*} 
Let $\psi_s=\epsilon_{n+1}+\cdots+\epsilon_{n+s}$, where $s \leq m$, then 
\begin{align*}
\langle \psi_s,\psi_s+2\rho \rangle = (n+m-2)s
\end{align*}
Moreover, for $u\in \mathbb{Z}$,
\begin{align*}
\langle u\phi_t, u\phi_t+2\rho \rangle &=ut(-t+n-m+u)\\
\langle u\psi_s,u\psi_s+2\rho \rangle &= us(s-n-m-u)
\end{align*}
\end{lemma}
\begin{proof}
This is a straightforward calculation.
\end{proof}

\begin{lemma}\label{tomato}
For a partition $\lambda\in \mathcal{P}_0$, let $\mathfrak{B}$ be the set of boxes in the rows $p+1$ and below, $\mathfrak{C}$ be the set of boxes in the columns $a+1$ and beyond. Then
\begin{align*}
\displaystyle\sum_{b\in\mathfrak{C}}(2c(b)-(a-p+b-q))=\displaystyle\sum_{b\in\mathfrak{B}}(2c(b)-(a-p+b-q))+qb(a+p).
\end{align*}
\end{lemma}

\begin{proof}
This is a straighforward calculation.
\end{proof}

The action of $z_0$ is as follows

\begin{theorem}
Let $T$ be a path $T=(T^{(0)},\dots,T^{(d)}=\lambda)$ in the Bratteli diagram, $\mathfrak{B}$ be the set of boxes in the rows $p+1$ and below in $T^{(0)}$, $v_T$ be the basis element in $\mathcal{L}^{\lambda}$ as defined prior to Theorem~\ref{z0}. 
\begin{align*}
z_0. v_T=
(qab+\displaystyle\sum_{b\in\mathfrak{B}}(2c(b)-(a-p+b-q)))v_T.
\end{align*}
\end{theorem}

\begin{proof}
Since the partition $(a^p)$ is a hook shape, either $n \geq p$ or $m\geq a$. Let us discuss by cases.

1) If $n<p$, then $m \geq a$. Therefore, the partition $(a^p)\cup (b^q)$ by pasting the rectangle $(b^q)$ right underneath $(a^p)$, is a hook shape and belongs to the set $\mathcal{P}_0$. Since $T^{(0)}\in \mathcal{P}$, $T^{(0)}$ can be viewed as the result of successively removing boxes from the lower rectangle in $(a^p)\cup (b^q)$, and adding them inside the rectangle to the right of $(a^p)$, to the spot that is the result of rotating the rectangle by $180$ degrees. Let ($(a^p)\cup (b^q)=\overline{\lambda^{(0)}},\overline{\lambda^{(1)}},\dots,\overline{\lambda^{(s-1)}},\overline{\lambda^{(s)}}=T^{(0)})$ be a sequence of the resulting diagrams from successively moving boxes in this fashion. Then by Lemma \ref{transfer},
\begin{align*}
\kappa_{\lambda^{(i)}}-\kappa_{\lambda^{(i-1)}}=4c(b)-2(a-p+b-q)
\end{align*}
where $b$ is the distinct box in $\lambda^{(i)}$.
Let $\mathfrak{C}$ be the set of boxes in $T^{(0)}$ in the columns $a+1$ and beyond. 
\begin{align*}
\displaystyle\sum_{i=1}^{s}(\kappa_{\lambda^{(i)}}-\kappa_{\lambda^{(i-1)}})=\displaystyle \sum_{b\in \mathfrak{C}}4c(b)-2(a-p+b-q)
\end{align*}
In other words, let $\lambda$ be the weight associated to the partition $T^{(0)}$, $\alpha\cup\beta$ be the weight associated to the partition $(a^p)\cup (b^q)$,
\begin{align*}
\kappa_{\lambda}-\kappa_{\alpha\cup\beta}=\sum_{b\in \mathfrak{C}}4c(b)-2(a-p+b-q)
\end{align*}
On the other hand, the action of $z_0$ is given by $\frac{1}{2}(\kappa_{M\otimes N}-\kappa_M-\kappa_N=\gamma_{M,N}$, and since $v_T\in L(\lambda)\otimes V^{\otimes d}$ (see the discussion before Theorem \ref{z0},) the eigenvalue of $z_0$ acting on $v_T$ is the eigenvalue of $\gamma_{M,N}$ acting on $L(\lambda)$, and by an argument similar to that in Lemma \ref{pierri}, this scalar is equal to 
\begin{align*}
\frac{1}{2}(\kappa_{\lambda}-\kappa_{\alpha}-\kappa_{\beta})
\end{align*} 
Hence the eigenvalue for $z_0$ is 
\begin{align*}
\frac{1}{2}(\kappa_{\alpha\cup\beta}-\kappa_{\alpha}-\kappa_{\beta})+\sum_{b\in \mathfrak{C}}(2c(b)-(a-p+b-q))
\end{align*}
To calculate $\kappa_{\alpha\cup\beta}$, $\kappa_{\alpha}$, $\kappa_{\beta}$, define $\phi_t=\epsilon_1+\cdots+\epsilon_t$, $t\leq n$, $\psi_s=\epsilon_1+\cdots+\epsilon_s$, $s\leq m$ in the notion of Lemma \ref{advinner}. Then 
\begin{align*}
\alpha&=a\phi_n+(p-n)\psi_a\\
\beta&=b\phi_n+(q-n)\psi_b\\
\alpha\cup \beta&= a\phi_n+(p-n)\psi_a+q\psi_b
\end{align*}
\begin{align*}
&\kappa_{\alpha\cup\beta}-\kappa_{\alpha}-\kappa_{\beta}\\
&=\langle a\phi_n,a\phi_n+2\rho\rangle + \langle (p-n)\psi_a+q\psi_b, (p-n)\psi_a+q\psi_b +2\rho \rangle\\
&- \langle a\phi_n , a\phi_n +2\rho \rangle - \langle (p-n)\psi_a, (p-n)\psi_a+2\rho \rangle\\
&- \langle b\phi_n , b\phi_n +2\rho \rangle - \langle (q-n)\psi_b, (q-n)\psi_b+2\rho \rangle\\
&=\langle (p-n)\psi_a, (p-n)\psi_a+2\rho \rangle + \langle q\psi_b, q\psi_b+2\rho \rangle +2\langle  (p-n)\psi_a, q\psi_b \rangle\\
&- \langle (p-n)\psi_a, (p-n)\psi_a+2\rho \rangle\\
&- \langle b\phi_n , b\phi_n +2\rho \rangle - \langle (q-n)\psi_b, (q-n)\psi_b+2\rho \rangle\\
&= \langle q\psi_b, q\psi_b+2\rho \rangle -2(p-n)qb \\
&- \langle b\phi_n , b\phi_n +2\rho \rangle - \langle (q-n)\psi_b, (q-n)\psi_b+2\rho \rangle\\
&= qb(b-n-m-q) -2(p-n)qb - bn(-n+n-m+b)\\
&-(q-n)b(b-n-m-(q-n))\\
&=-2bpq
\end{align*}
Hence $z_0$ acts via 
\begin{align*}
-bpq + \sum_{b\in \mathfrak{C}}(2c(b)-(a-p+b-q))
\end{align*}
Due to Lemma \ref{tomato}, this is equal to 
\begin{align*}
abq + \sum_{b\in \mathfrak{B}}(2c(b)-(a-p+b-q))
\end{align*}
where $\mathfrak{B}$ is the set of boxes in rows $p+1$ and below in $T^{(0)}$.\\

2) If $p\leq n<p+q$, it is still required that $m\geq a$. $(a^p)\cup(a^q)$ is still a hook shape and is contained in $\mathcal{P}_0$. Similar to an argument above, $z_0$ acts on $v_T$ via a scalar equal to
\begin{align*}
\frac{1}{2}(\kappa_{\alpha\cup\beta}-\kappa_{\alpha}-\kappa_{\beta})+\sum_{b\in \mathfrak{C}}(2c(b)-(a-p+b-q))
\end{align*}
In this case,
\begin{align*}
\alpha&=a\phi_n+(p-n)\psi_a\\
\beta&=b\phi_q\\
\alpha\cup\beta&=a\phi_n+(p-n)\psi_a+q\psi_b
\end{align*}
\begin{align*}
&\kappa_{\alpha\cup\beta}-\kappa_{\alpha}-\kappa_{\beta}\\
&=\langle a\phi_n, a\phi_n+2\rho \rangle + \langle (p-n)\psi_a, (p-n)\psi_a + 2\rho \rangle \\
&+ \langle q\psi_b, q\psi_b+2\rho \rangle + 2 \langle (p-n)\psi_a, q\psi_b \rangle\\
&- \langle a\phi_n, a\phi_n+2\rho \rangle -\langle (p-n)\psi_a, (p-n)\psi_a+2\rho \rangle\\
&- \langle b\phi_q,b\phi_q+2\rho \rangle\\
&= \langle q\psi_b, q\psi_b+2\rho \rangle -2 (p-n)qb - \langle b\phi_q,b\phi_q+2\rho \rangle\\
&= qb(b-n-m-q)-2bq(p-n)-bq(-q+n-m+b)\\
&=-2bpq
\end{align*}
The rest of the argument is the same to that in case 1).\\

3) If $n\geq p+q$. In this case, all diagrams $\lambda$ with Littlewood-Richardson coefficient $c^{\lambda}_{(a^p),(b^q)}$ are hook shapes and belong to $\mathcal{P}_0$. In particular, the partition $(a^p)+(b^q)$ that is the result of putting the two rectangles side by side horizontally, is in $\mathcal{P}_0$. Similar to case 1), the diagram $T^{(0)}$ can be viewed as the result of successively moving boxes in $(a^p)+(b^q)$ to the rows $p+1$ and below. Let $\lambda$ be the weight associated to $T^{(0)}$ and $\alpha+\beta$ the weight associated to $(a^p)+(b^q)$. By taking a sequence of the intermediate diagrams and apply Lemma \ref{transfer},
\begin{align*}
\kappa_{\lambda}-\kappa_{\alpha+\beta}=\displaystyle\sum_{b\in\mathfrak{B}}(4c(b)-2(a-p+b-q))
\end{align*}
The eigenvalue of $z_0$ is 
\begin{align*}
&\frac{1}{2}(\kappa_{\lambda}-\kappa_{\alpha}-\kappa_{\beta})\\
&=\frac{1}{2}(\kappa_{\alpha+\beta}-\kappa_{\alpha}-\kappa_{\beta})+\displaystyle\sum_{b\in\mathfrak{B}}(2c(b)-(a-p+b-q))
\end{align*}
Since $\alpha=a\phi_p$, $\beta=b\phi_q$, $\alpha+\beta=a\phi_p+b\phi_q$,
\begin{align*}
\kappa_{\alpha+\beta}-\kappa_{\alpha}-\kappa_{\beta}=2\langle a\phi_p, b\phi_q\rangle =2abq
\end{align*}
and the conclusion follows.
\end{proof}

\subsection{Restriction as $\mathcal{H}^{\operatorname{ext}}_{d}$-modules}

Define new elements $w_i=z_i-\frac{1}{2}(a-p+b-q)$, $1\leq i \leq d$, $w_0=z_0$. In \cite[Theorem 4.3]{D}, Daugherty introduced a second presentation of $\mathcal{H}^{\operatorname{ext}}_d$, using generators $x_1$, $w_0$, \dots, $w_d$, $t_1$, \dots, $t_{d-1}$. In particular, the newly defined elements act by $w_i.v_{T}=c_T(i)$, where

\begin{align*}
c_T(0)&=
qab+\displaystyle\sum_{b\in\mathfrak{B}}2(c(b)-\frac{1}{2}(a-p+b-q)))\\
c_T(i)&=c(T^{(i)}/T^{(i-1)})-\frac{1}{2}(a-p+b-q), \hspace{.5 in} {1\leq i\leq d}
\end{align*}

As explained in \cite[Lemma 4.14]{D}, $c(\lambda/\mu)\neq c(\lambda'/\mu)$ for any two partitions $\lambda\neq \lambda'$ that both contain $\mu$. In particular, the sequence of integers $c_T(0),\dots,c_T(d)$ uniquely determines the path $T$.

Using the new presentation, Daugherty constructed seminormal represesntations as follows. First define the action of $s_i (0\leq d)$ on the set of paths $\Gamma^{\lambda}$ to be 
\begin{align*}
s_i.T=\begin{cases}
\text{the unique other path }T' \text{ such that} (T')^{(j)}=T^{(j)}, \forall j\neq i, \text{ if }T' \text{ exists}\\
T \text{ otherwise}
\end{cases}
\end{align*}
By the discussion in \cite[Section~4.3]{D}, this action is well-defined. 

\begin{theorem}(\cite[Theorem~4.15]{D})
Let $D^{\lambda}$ be the vector space spanned by $\{v_T\}_{T\in \Gamma^{\lambda}}$. Let the generators act on $D^{\lambda}$ as follows.
\begin{align}
x_1.v_T&=[x_1]_{T,T}v_T+[x_1]_{T,s_0.T}v_{s_0.T} \notag \\
s_i.v_T&=[s_i]_{T,T}v_T+[s_i]_{T,s_i.T}v_{s_i.T},\hspace{.2 in}1\leq i\leq d \notag \\
w_i.v_T&=c_T(i)v_T \hspace{.2 in} 0\leq i\leq d. \label{wiacts}
\end{align}
Here, $[x_1]_{T,T}$, $[x_1]_{T,s_0.T}$, $[s_i]_{T,T}$, $[s_i]_{T,s_i.T}$ are constants for all $T\in \Gamma^{\lambda}$, with $[x_1]_{T,s_0.T}=0$ if $T=s_0.T$, $[s_i]_{T,s_i.T}=0$ if $T=s_i.T$, and they satisfy the list of conditions in  \cite[Theorem~4.15]{D}. Then this is a well-defined $\mathcal{H}^{\operatorname{ext}}_d$-module structure on $D^{\lambda}$.
\end{theorem}

\begin{remark}
The proof of the above theorem in \cite{D} in fact shows that, based on the relations in $\mathcal{H}^{\operatorname{ext}}_d$, Equation~\ref{wiacts} determines the action of $x_1$ and $s_i$, such that they have to obey the construction up the choice of constants $[x_1]_{T,T}$, $[x_1]_{T,s_0.T}$, $[s_i]_{T,T}$, $[s_i]_{T,s_i.T}$ satisfying the mentioned conditions.
\end{remark}

% that classify all irreducible represetations of $\mathcal{H}^{\operatorname{ext}}_d$ upon which $w_i$ acts as above for every $i$.

The image of $\mathcal{H}^{\operatorname{ext}}_{d}$ under $\Psi$ is a large subalgebra of $\mathcal{H}_d$ under the following sense.

\begin{theorem}
For every $\lambda\in \mathcal{P}_d$, the module $\operatorname{Res}^{\mathcal{H}_d}_{\mathcal{H}^{\operatorname{ext}}_d}\mathcal{L}^{\lambda}$ is irreducible.
\end{theorem}
\begin{proof}
The action of $w_i$ already agrees with that in \cite[Theorem~4.15]{D}, it remains to show that the rest of the generators also act according to the same construction. In particular, since $x_1w_i=w_ix_1$ for all $2\leq i\leq d$, $x_1$ preserves all eigenspaces of $w_i(i\geq 2)$. When $\lambda=T^{(d)}$ is given, the eigenvalues of $w_d$, or equivalently the contents $c(T^{(d)}/\mu)$ are all distints for any diagram $\mu$ of rank $d-1$ that has an edge to $T^{(d)}$, therefore the content $c_T(d)$ completely determines the diagram $T^{(d-1)}$ in the path $T$. Similarly, the sequence of integers $c_T(2),\dots,c_T(d)$ completely determines the diagrams $T^{(1)},\dots,T^{(d)}=\lambda$ and the path from  $T^{(1)}$ to $\lambda$. According to \cite[Example~2.7]{D}, \cite[Lemma~3.3]{Stan} or \cite[Theorem~2.4]{Ok}, there are at most two diagrams $\gamma \in \mathcal{P}_0$ that have an edge to $T^{(1)}$. Denote the two resulted paths as $T$ and $s_0.T$, it follows that
\begin{align*}
x_1.v_T=[x_1]_{T,T}v_T+[x_1]_{T,s_0.T}v_{s_0.T}
\end{align*}
which agrees with the construction in \cite[Theorem~4.15]{D}.

Similarly, the action of $s_i$ satisfies $s_iw_j=w_js_i$, $\forall j\neq i,i+1$. Therefore $s_i$ fixes a subspace spanned by all paths with diagrams $T^{(0)}$, \dots, $T^{(i)}$ in the beginning and $T^{(i+2)}$, \dots, $T^{(d)}=\lambda$ in the end, where the diagrams are determined by the content $c_T(0)$, \dots $c_T(i-1)$, $c_T(i)$, \dots, $c_T(d)$ as before, with at most twopartitions for the choice of $T^{(i)}$ in the middle. Denote the resulted two paths by $T$ and $s_i.T$, we have 
\begin{align*}
s_i.v_T=[s_i]_{T,T}v_T+[s_i]_{T,s_0.T}v_{s_0.T}
\end{align*}
which also agrees with the construction in \cite[Theorem~4.15]{D}. Since the action is well-defined, these constants automatically satisfy the relations specified in the theorem, therefore defines an irreducible representation.

\end{proof}

\makeatletter
\renewcommand*{\@biblabel}[1]{\hfill#1.}
\makeatother


\begin{thebibliography}{10}
\bibitem{AS}Arakawa, Tomoyuki; Suzuki, Takeshi. Duality between $\mathfrak{sl}_n(\mathbb{C})$ and the degenerate affine Hecke algebra. \emph{J. Algebra} \textbf{209} (1998), no. 1, 288-304. 
\bibitem{BK} Brundan, Jonathan; Kujawa, Jonathan. A new proof of the Mullineux conjecture. \emph{J. Algebraic Combin.} \textbf{18} (2003), no. 1, 13-39.
\bibitem{BR} Berele, A.; Regev, A.
Hook Young diagrams, combinatorics and representations of Lie superalgebras. 
\emph{Bull. Amer. Math. Soc.} (N.S.) \textbf{8} (1983), no. 2, 337-339. 
\bibitem{D} Daugherty, Zajj.
Degenerate two-boundary centralizer algebras, Doctoral dissertation, University of Wisconsin-Madison, June 2010

\bibitem{Kac}
Kac, V. G.
Characters of typical representations of classical Lie superalgebras. 
\emph{Comm. Algebra} \textbf{5} (1977), no. 8, 889-897. 
\bibitem{Kuj}Kujawa, Jonathan. Crystal structures arising from representations of $GL(m|n)$. \emph{Represent. Theory} \textbf{10} (2006), 49-85. 
\bibitem{Ok}Okada, S. Applications of Minor Summation Formulas to Rectangular-Shaped Representations of Classical Groups, \emph{J. Algebra} \textbf{205} (1998), 337-367.

\bibitem{Rem} Remmel, Jeffrey B. The combinatorics of $(k,l)$-hook Schur functions. \emph{Combinatorics and algebra (Boulder, Colo., 1983)}, 253-287, \emph{Contemp. Math.}, \textbf{34}, Amer. Math. Soc., Providence, RI, 1984.

\bibitem{Ser1}A. Sergeev, Enveloping algebra centre for Lie superalgebras $GL$ and $Q$, PhD thesis, Moscow State
University, Moscow, 1987.

\bibitem{Ser2}A. Sergeev, The invariant polynomials on simple Lie superalgebras, \emph{Represent. Theory} \textbf{3} (1999),
250-280.

\bibitem{Ser3}A. Sergeev, Tensor algebra of the identity representation as a module over the Lie superalgebras
$GL(n,m) and Q(n)$, \emph{Math. USSR Sbornik} \textbf{51} (1985), 419–427.

\bibitem{Stem} Stembridge, J. R. Multiplicity-free products and restrictions of Weyl characters. \emph{Represent. Theory} \textbf{7} (2003), 404-439.
\bibitem{Stan}Stanley, R.P. Symmetries of plane partitions, \emph{J. Combin. Theory Ser. A}, \textbf{46} (1986)
103-113.

\end{thebibliography}
\end{document}